\documentclass[aap,preprint]{imsart}
\usepackage{amsthm,amsmath,amsfonts,amssymb}

\startlocaldefs
\newcommand{\dd}{\mathrm{d}}
\newcommand{\bbR}{\mathbb{R}}
\newcommand{\tsfrac}[2]{{\textstyle \frac{#1}{#2}}}
\newcommand{\twobyone}[2]{\begin{array}{c} #1 \\ #2 \end{array}}
\newcommand{\twobytwo}[4]{\begin{array}{cc} #1 & #2 \\ #3 & #4 \end{array}}

\newtheorem{theorem}{Theorem}[section]
\newtheorem{lemma}[theorem]{Lemma}
\newtheorem{corollary}[theorem]{Corollary}

\endlocaldefs

\begin{document}

\begin{frontmatter}

\title{Efficiency and computability of MCMC with Langevin, Hamiltonian, and other matrix-splitting proposals}
\runtitle{Efficiency of MCMC}


\begin{aug}
\author{\fnms{Richard A.} \snm{Norton}\corref{}\ead[label=e1]{richard.norton@otago.ac.nz}}
\and
\author{\fnms{Colin} \snm{Fox}\ead[label=e2]{colin.fox@otago.ac.nz}}

\runauthor{R. A. Norton and C. Fox}

\affiliation{University of Otago}
\address{Department of Physics\\University of Otago\\PO Box 56\\Dunedin 9054\\ New Zealand\\ \printead{e1,e2}}
\end{aug}

\begin{abstract}
We analyse computational efficiency of Metropolis-Hastings algorithms with AR(1) process proposals.  These proposals include, as a subclass, discretized Langevin diffusion (e.g. MALA) and discretized Hamiltonian dynamics (e.g. HMC).  

By including the effect of Metropolis-Hastings we extend earlier work by Fox and Parker, who used matrix splitting techniques to analyse the performance and improve efficiency of AR(1) processes for targeting Gaussian distributions.

Our research enables analysis of MCMC methods that draw samples from non-Gaussian target distributions by using AR(1) process proposals in Metropolis-Hastings algorithms, by analysing the matrix splitting of the precision matrix for a local Gaussian approximation of the non-Gaussian target.
\end{abstract}

\begin{keyword}[class=MSC]
\kwd[Primary ]{60J22, 65Y20}
\kwd{}
\kwd[; secondary ]{60J05, 62M05, 65C40, 68Q87, 68W20}
\end{keyword}

\begin{keyword}
\kwd{Markov chain, Monte Carlo, sampling, Gaussian, multivariate normal distribution, AR(1) process, Metropolis-Hastings algorithm, Metropolis-adjusted Langevin algorithm, Hybrid Monte Carlo}
\kwd{}
\end{keyword}

\end{frontmatter}

\section{Introduction}
\label{sec intro}

Many existing Metropolis-Hastings (MH) algorithms for sampling from a non-Gaussian target distribution $\pi$ use an AR(1) process proposal; given current state $x \in \bbR^d$ the proposal $y \in \bbR^d$ is given by
\begin{equation}
\label{eq ar1}
	y = G x + g + \nu
\end{equation}
where $G \in \bbR^{d\times d}$ is the iteration matrix, $g \in \bbR^d$ is a fixed vector and $\nu$ is an i.i.d. draw from $N(0,\Sigma)$.  In general, $G$, $g$ and $\Sigma$ may depend on $x$.  The MH acceptance probability is
$$
	\alpha(x,y) = 1 \wedge \frac{\pi(y) q(y,x)}{\pi(x) q(x,y)}
$$
where $\pi(x)$ denotes the target probability density function and $q(x,\dd y) = q(x,y) \dd y$ is the transition kernel for the proposal $y$ given current state $x$.  

An important case of \eqref{eq ar1} is when the AR(1) process comes from a matrix splitting of some matrix $\mathcal{A}$ associated with $\pi$ at $x$, for example when $\mathcal{A}$ is the local Hessian of $\log \pi$, or its inverse, or an approximation to these matrices.  In this case we find it more convenient and revealing to write \eqref{eq ar1} in terms of the matrix splitting $\mathcal{A} = M-N$ so that $y$ is given by solving
\begin{equation}
\label{eq ar1b}
	M y = N x + \beta + \nu
\end{equation}
where $\beta$ is a vector and $\nu$ is an i.i.d. draw from $N(0,M^T+N)$, $G = M^{-1}N$, $g = M^{-1}\beta$ and $\Sigma = M^{-1}(M^T + N) M^{-T}$.  In fact, if the spectral radius of $G$ is less than $1$ then the AR(1) processes \eqref{eq ar1} can be written in terms of a matrix splitting and vice versa, see Section \ref{sec equiv}.  Moreover, if the spectral radius of $G$ is less than $1$, then the proposal chain generated by \eqref{eq ar1b} will converge to $N(\mathcal{A}^{-1}\beta,\mathcal{A}^{-1})$, which we call the \emph{proposal target distribution}, see \cite{FP2015}.  

Later, we see that the Metropolis-adjusted Langevin algorithm (MALA) is an example of an AR(1) process proposal with a corresponding matrix splitting, as are other discretized Langevin diffusion proposals, as are also the Hybrid Monte Carlo algorithm (HMC) and other discretized Hamiltonian dynamics proposals.  We will use this observation to further analyse these methods and discuss their computability in Section \ref{sec examples}.

Although identifying \eqref{eq ar1} with \eqref{eq ar1b} is useful for analysing existing methods, if the task is designing a proposal for the MH algorithm then using \eqref{eq ar1b} is more natural because we can begin by choosing $\mathcal{A}$ and $\beta$.  For example, we are particularly interested in the proposal where $\mathcal{A}$ and $\beta$ are chosen such that $N(\mathcal{A}^{-1} \beta,\mathcal{A}^{-1})$ is a local Gaussian approximation to $\pi$, i.e. $-\tsfrac{1}{2} x^T \mathcal{A} x + \beta^T x$ is a local quadratic approximation to $\log \pi$.  

After selecting $\mathcal{A}$ and $\beta$, we must also choose the splitting $M$ and $N$.  This choice will effect the computational cost of each proposal, as well as how fast the AR(1) process will converge to the proposal target distribution.

We may use one or more iterations of the AR(1) process \eqref{eq ar1b} as a proposal for the MH algorithm.  For the next proposal (if the previous one was accepted) we then use a new local Gaussian approximation to $\pi$.  This idea mimics the design of some optimizers where a local quadratic approximation to the objective function is minimized at each iteration, see for example \cite{NW1999}.  

Our concern is the analysis of Markov chain Monte Carlo (MCMC) methods with AR(1) process proposals which fall into three cases.  
\begin{enumerate}
\item The target and proposal target distributions are the same, i.e. $\pi = N(\mathcal{A}^{-1}\beta,\mathcal{A}^{-1})$ and the AR(1) process targets $\pi$ directly.  Then \eqref{eq ar1} is a fixed-scan Gibbs sampler for $\pi$, and the accept/reject step in the MH algorithm is redundant.  Fox and Parker \cite{F2013,FP2015,FP2014} have already studied this case. They showed how to choose $M$ and $N$ (equivalently choose $G$ and $\Sigma$) to obtain an efficient algorithm and they also showed how to accelerate the AR(1) process to achieve even greater efficiency.
\item The target distribution $\pi$ is normal, $N(A^{-1}b,A^{-1})$ for some matrix $A$ and vector $b$, but not targeted by the AR(1) process, i.e. $A \neq \mathcal{A}$ and/or $b \neq \beta$.  Then a MH accept/reject step is required to ensure $\pi$ is targeted.
\item The target distribution $\pi$ is not normal, hence the AR(1) process does not target $\pi$, and a MH accept/reject step is required to ensure $\pi$ is targeted.  The proposal target is a local Gaussian approximation to $\pi$.  
\end{enumerate}

Cases 2 and 3 where the target and proposal target distributions are different are harder to analyse than case 1 as the MH accept/reject step affects the transition kernel.  This begs the following questions in these cases: 
\begin{itemize}
\item
How should we choose $\mathcal{A}$, $\beta$, $M$ and $N$ to construct a computationally efficient algorithm?
\item
Can we accelerate MH algorithms with AR(1) process proposals for non-Gaussian target distributions?
\end{itemize}

We measure the quality of a sampling algorithm by its computational efficiency, which is a combination of the compute time required for the Markov chain to reach equilibrium (burn in), and once in equilibrium, the compute time required to produce quasi-independent samples.  Simply, we are interested in getting the most `bang for buck', where `bang' is the number of quasi-independent samples and `buck' is the computing cost.

Fox and Parker \cite{FP2015} mainly considered distributional convergence from an arbitrary state, so they measured burn in.  Here we are concerned about the compute time to produce quasi-independent samples.  The integrated autocorrelation time (length of Markov chain with variance reducing power of one independent sample) is a proxy for independence, hence we try to measure compute time per integrated autocorrelation time.

In this article we focus on analysing case 2 above where $\pi$ is the normal $N(A^{-1}b,A^{-1})$, but not the same as the proposal target $N(\mathcal{A}^{-1}\beta,\mathcal{A}^{-1})$.  However, our results are relevant to case 3 where $\pi$ is not normal because the cases share important features.  In particular, the accept/reject step in the MH algorithm is used in both cases to correct the difference between the target and proposal target distributions.  The simplification of only considering case 2 makes the analysis of the transition kernel possible, whilst keeping the essential feature of the non-Gaussian target case, that the spectrum of the transition kernel has been changed by the MH algorithm.  If we cannot analyse or accelerate case 2 then we do not expect to be able to analyse or accelerate case 3.  

Moreover, in case 3 we might construct an inhomogeneous Markov chain by updating the proposal rule at each iteration (or every few iterations) by updating the local Gaussian approximation of the target distribution, hence we are interested in the local behaviour of the algorithm where the target distribution is approximately normal.

Our analysis of case 2 is for the special case when the matrices $M$ and $N$ are functions of $A$.  This allows us to simultaneously diagonalise both the AR(1) process and the target distribution with a change of variables.  This is not an overly restrictive condition if factorizing the matrix $A$ is infeasible, and we will see below that it includes several important examples of MH algorithms already in use.

Our analysis is also analogous to the analysis of optimizers where it is useful to test the behaviour of optimizers on quadratic cost functions (in the probability sense, a Gaussian is equivalent to a quadratic function since $\log \pi$ is quadratic when $\pi$ is Gaussian).

Acceleration techniques for AR(1) processes for sampling in case 1 (normal target) do not necessarily accelerate sampling in case 3 (non-normal target) when the accelerated AR(1) process is used as a proposal in the MH algorithm.  Goodman and Sokal \cite{GS1989} accelerated Gibbs sampling of normal distributions using ideas from multigrid linear solvers, but only observed modest efficiency gains in their non-normal examples (exponential distributions with fourth moments).  Green and Han \cite{GH1992} applied successive-over-relaxation to a local Gaussian approximation of a non-Gaussian target distribution as a proposal for the MH algorithm.  Again, they did not observe significant acceleration in the non-normal target case.

By recognising that MH algorithms with discretized Langevin diffusion or discretized Hamiltonian dynamics proposals are examples of MH algorithms with AR(1) process proposals, we are able to apply our theory to these MCMC methods.  We find that we can replicate existing results in the literature for these methods, and we can also extend existing results in some cases, see Section \ref{sec examples}.  We also discuss the computational efficiency of these methods.

Strictly, our analysis is exact for Gaussian target distributions, but it can also be applied to the case when $\pi$ is absolutely continuous with respect to a Gaussian (see e.g. \cite{S2010}, \cite{BRS2009} or \cite{BRSV2008}).  

Sampling from a Gaussian (or a target that is absolutely continuous with respect to a Gaussian) is also a problem that arises in inverse problems where the forward model has the form $y = F x + n$ where $y$ is the observed data, $F$ is a linear (or bounded for absolutely continuous with respect to a Gaussian) operator, $x$ is the unknown image with Gaussian prior distribution, and $n$ is Gaussian noise, see e.g. \cite[Example 6.23]{S2010}.  If the posterior is Gaussian with hyperparameters satisfying some other marginal distribution, and we can sample the hyperparameters independently, then conditionally sampling the posterior given the hyperparameters is also a Gaussian sampling problem \cite{B2012}.  Other applications that involve sampling Gaussian distributions include Brownian bridges \cite{BRSV2008}.

For the purpose of evaluating the cost of computing a proposal from \eqref{eq ar1b} for a given matrix splitting, we make similar assumptions as made in numerical linear algebra for solving $A x = b$, since solver and sampling algorithms share many of the same operations.  We assume that we can efficiently compute matrix-vector products $A v$ for any vector $v \in \bbR^d$; for example $A$ may be sparse.  Furthermore, we assume that $d$ is sufficiently large so that it is infeasible to directly compute $A^{1/2}$ or any other matrix factorization of $A$ and then directly compute independent samples from our Gaussian target.  

The remaining sections are as follows.  In Section \ref{sec equiv} we quickly show that \eqref{eq ar1} and \eqref{eq ar1b} are equivalent, then Sections \ref{sec expect} and \ref{sec jumpsize} present new analyses for the expected acceptance rate and jump size of MH algorithms with AR(1) process proposals.  Section \ref{sec examples} then applies this new analysis to proposals from Langevin diffusion and Hamiltonian dynamics.  We see that these proposals are AR(1) processes and we identify the corresponding matrix splitting and proposal target distribution.  Using our earlier analysis we assess the convergence properties of these methods as $d \rightarrow \infty$.  We provide concluding remarks in Section \ref{sec conclusion}.

\section{AR(1) processes correspond to matrix splittings}
\label{sec equiv}

We can express an AR(1) process using either \eqref{eq ar1} or \eqref{eq ar1b}, provided the AR(1) process converges.

\begin{theorem}
If we are given $G$, $g$ and $\Sigma$, and the spectral radius of $G$ is less than $1$, then the AR(1) process \eqref{eq ar1} can be written as \eqref{eq ar1b} using
\begin{equation}
\label{eq calA}
	\mathcal{A} = \left( \sum_{l=0}^\infty G^l \Sigma (G^{T})^l \right)^{-1}
\end{equation}
and
\begin{equation}
\label{eq MNbeta}
\begin{split}
	M &= \mathcal{A} (I-G)^{-1}, \\
	N &= \mathcal{A} (I-G)^{-1} G, \qquad \qquad 
	 \beta = \mathcal{A} (I-G)^{-1} g.
\end{split}
\end{equation}
Note that we also have $\mathcal{A} = M - N$ symmetric and positive definite.
\end{theorem}

\begin{proof}
Since the spectral radius of $G$ is less than $1$ and $\Sigma$ is symmetric positive definite, it follows that $\mathcal{A}^{-1} := \sum_{l=0}^\infty G^l \Sigma (G^{T})^l$ is well-defined, symmetric and positive definite.

It is then easy to see that \eqref{eq calA} and \eqref{eq MNbeta} satisfy $\mathcal{A} = M-N$, $G = M^{-1} N$ and $g = M^{-1} \beta$.  We must also check that $\Sigma = M^{-1} (M^T + N) M^{-T}$ is satisfied.  Substituting \eqref{eq MNbeta} into $M^{-1} (M^T+N)M^{-T}$ we get
\begin{align*}
	M^{-1}(M^T + N) M^{-T} 
	&= (I-G) \mathcal{A}^{-1} + G \mathcal{A}^{-1} G^T \\
	&= \mathcal{A}^{-1} - G \mathcal{A}^{-1} G^T \\
	&= \sum_{l=0} G^l \Sigma (G^T)^l - \sum_{l=1}^\infty G^l \Sigma (G^T)^l \\
	&= \Sigma.
\end{align*}
\end{proof}

In the following special case we obtain a symmetric matrix splitting.  

\begin{corollary}
\label{lem equiv}
If the spectral radius of $G$ is less than $1$ and $G \Sigma$ is symmetric, then the AR(1) process \eqref{eq ar1} has a corresponding matrix splitting defined by
\begin{align*}
	M &= \Sigma^{-1} (I + G), & \mathcal{A} &= M (I-G) = \Sigma^{-1} (I-G^2), \\
	N &= MG = \Sigma^{-1}(I+G)G, & \beta &= Mg = \Sigma^{-1}(I+G) g,
\end{align*}
and $M$ and $N$ are symmetric (we say the matrix splitting is symmetric).
\end{corollary}

\begin{proof}
These matrix splitting formulae follow from the identity 
$$
	\sum_{l=0}^\infty G^l \Sigma (G^T)^l = \sum_{l=0}^\infty G^{2l} \Sigma = (I-G^2)^{-1} \Sigma.
$$
To see that $M$ is symmetric (and hence also $N$ since $\mathcal{A}$ is symmetric and $\mathcal{A} = M-N$) we note that
$$
	M = \left( (I-G) \sum_{l=0}^\infty G^{2l} \Sigma \right)^{-1}.
$$
\end{proof}

\section{Expected acceptance rate for a general matrix splitting}
\label{sec expect}

The expected acceptance rate is a quantity that is related to efficiency and for optimal performance the proposal is usually tuned so that the observed average acceptance rate is well away from $0$ and $1$.  For example, it has been shown in the case when $d \rightarrow \infty$ that $0.234$ is optimal for the random-walk Metropolis algorithm (RWM) \cite{RGG1997}, $0.574$ is optimal for MALA \cite{RR1998}, and $0.651$ is optimal for HMC \cite{BPRSS2010}.  All of these results required expressions for the expected acceptance rate of the algorithm as $d \rightarrow \infty$.  Here we derive an expression for the expected acceptance rate for an AR(1) process proposal \eqref{eq ar1b} with Gaussian target $N(A^{-1}b,A^{-1})$, provided the splitting matrices are functions of $A$.  Thus, our MH algorithm is defined by
\begin{equation}
\label{MH0}
\begin{split}
	\mbox{Target:} & \qquad N(A^{-1}b, A^{-1}), \\
	\mbox{Proposal:} & \qquad y = G x + M^{-1} \beta + (\mathcal{A}^{-1} - G \mathcal{A}^{-1} G^T)^{1/2} \xi, 
\end{split}
\end{equation}
where $\xi \sim N(0,I)$, and we have used $G = M^{-1}N$ and $M^{-1}(M^T+N)M^{-T} = \mathcal{A}^{-1} - G\mathcal{A}^{-1} G^T$ \cite[Lem. 2.3]{FP2014}.  The following lemma is a result of simple algebra.  The proof is in the Appendix.

\begin{lemma}
\label{lem1}
Suppose $\mathcal{A} = M-N$ is a symmetric splitting.  Then the MH acceptance probability for \eqref{MH0} satisfies
$$
	\alpha(x,y) = 1 \wedge \exp \left( -\frac{1}{2} y^T (A - \mathcal{A}) y + \frac{1}{2} x^T(A-\mathcal{A}) x + (b-\beta)^T (y-x)\right).
$$
\end{lemma}

Now define a spectral decomposition
$$
	A = Q \Lambda Q^T
$$ 
where $Q \in \bbR^{d\times d}$ is orthogonal and $\Lambda = \operatorname{diag}(\lambda_1^2,\dotsc,\lambda_d^2)$ is a diagonal matrix of eigenvalues of $A$.  Although we may not be able to compute $Q$ and $\Lambda$ this does not stop us from using the spectral decomposition for theory.  Simple algebra gives us the following result.

\begin{lemma}
\label{lem2}
Suppose $M = M(A)$ and $N = N(A)$ are functions of $A$.  Then $G$ and $\mathcal{A}$ are also functions of $A$ and under the coordinate transformation
$$
	x \leftrightarrow Q^T x
$$
the MH algorithm \eqref{MH0} is transformed to a MH algorithm defined by
\begin{equation}
\label{MH1}
\begin{split}
	\mbox{Target:} & \qquad N(\Lambda^{-1} Q^T b, \Lambda^{-1}), \\
	\mbox{Proposal:} & \qquad y = G x + M(\Lambda)^{-1} Q^T \beta + (\tilde{\Lambda}^{-1} - G \tilde{\Lambda}^{-1} G^T)^{1/2} \xi, 
\end{split}
\end{equation}
where $\xi \sim N(0,I)$, and $G = M(\Lambda)^{-1}N(\Lambda)$ and $\tilde{\Lambda} = \mathcal{A}(\Lambda)$ are diagonal matrices.  
\end{lemma}

Using Lemma \ref{lem1} we see that the acceptance probability of MH algorithms \eqref{MH0} and \eqref{MH1} are identical and hence it is sufficient to analyse the convergence properties of \eqref{MH1} to determine the convergence properties of \eqref{MH0}.  

We will need the following Lyapunov version of the Central Limit Theorem, see e.g. \cite[Thm. 27.3]{billingsley1995}

\begin{theorem}
\label{thm clt}
Let $X_1,\dotsc,X_d$ be a sequence of independent random variables with finite expected value $\mu_i$ and variance $\sigma_i^2$.  Define
$
	s_d^2 := \sum_{i=1}^d \sigma_i^2.
$
If there exists a $\delta > 0$ such that
$$
	\lim_{d \rightarrow \infty} \frac{1}{s_d^{2+\delta}} \sum_{i=1}^d \mathrm{E}[ |X_i - \mu_i|^{2+\delta} ] = 0,
$$
then 
$$
	\frac{1}{s_d} \sum_{i=1}^d (X_i - \mu_i) \xrightarrow{\mathcal{D}} N(0,1) \qquad \mbox{as $d \rightarrow \infty$}.
$$
\end{theorem}

An equivalent conclusion to this theorem is $\sum_{i=1}^d X_i \rightarrow N(\sum_{i=1}^d \mu_i, s_d^2)$ in distribution as $d \rightarrow \infty$.  Another useful fact is  
\begin{equation}
\label{eq uf1}
	X \sim N(\mu,\sigma^2) \qquad \Rightarrow \qquad \mathrm{E}[ 1 \wedge \mathrm{e}^X ] = \Phi(\tsfrac{\mu}{\sigma}) + \mathrm{e}^{\mu + \sigma^2/2} \Phi(-\sigma - \tsfrac{\mu}{\sigma})
\end{equation}
where $\Phi$ is the standard normal cumulative distribution function.  See e.g. \cite[Prop. 2.4]{RGG1997} or \cite[Lem. B.2]{BRS2009}.

\begin{theorem}
\label{thm accept}
Suppose that $M$ and $N$ in \eqref{MH0} are functions of $A$, and the Markov chain is in equilibrium, i.e. $x \sim N(A^{-1}b,A^{-1})$.  Recall that $\lambda_i^2$ are eigenvalues of $A$, and define $\tilde{\lambda}^2_i$ and $G_i$ to be the eigenvalues of $\mathcal{A}$ and $G$ respectively.  Also define
\begin{align*}
	\tilde{g}_i &:= 1-G_i, & g_i &:= 1-G_i^2, & m_i &:= (A^{-1}b)_i, & \tilde{m}_i &:= (\mathcal{A}^{-1}\beta)_i, \\
	r_i &:= \frac{\lambda_i^2-\tilde{\lambda}_i^2}{\lambda_i^2}, &
	\tilde{r}_i &:= \frac{\lambda_i^2-\tilde{\lambda}_i^2}{\tilde{\lambda}_i^2}, & 
	\hat{r}_i &:= m_i - \tilde{m}_i
\end{align*}
and 
\begin{align*}
	T_{0i} &:= \hat{r}_i^2 \lambda_i^2 ( \tsfrac{1}{2} r_i g_i - \tilde{g}_i), \\
	T_{1i} &:= \hat{r}_i \lambda_i ( r_i g_i - \tilde{g}_i), \\
	T_{2i} &:= \hat{r}_i \lambda_i g_i^{1/2} (1+\tilde{r}_i)^{1/2} ( 1 - r_i G_i), \\
	T_{3i} &:= \tsfrac{1}{2} r_i g_i, \\
	T_{4i} &:= -\tsfrac{1}{2} r_i g_i (1 + \tilde{r}_i), \\
	T_{5i} &:= -r_i G_i g_i^{1/2} (1+\tilde{r}_i)^{1/2}.
\end{align*}
If there exists a $\delta > 0$ such that
\begin{equation}
\label{eq Tcond}
	\lim_{d \rightarrow \infty} \frac{\sum_{i=1}^d |T_{ji}|^{2+\delta}}{ \left( \sum_{i=1}^d |T_{ji}|^2 \right)^{1+\delta/2}} = 0 \qquad \mbox{for $j=1,2,3,4,5$}
\end{equation}
($j=0$ is not required), then 
$$
	Z:= \log \left( \frac{\pi(y)q(y,x)}{\pi(x)q(x,y)} \right) \xrightarrow{\mathcal{D}} N(\mu,\sigma^2) \qquad \mbox{as $d \rightarrow \infty$}
$$
where $\mu = \lim_{d\rightarrow \infty} \sum_{i=1}^d \mu_i$ and $\sigma^2 = \lim_{d\rightarrow \infty} \sum_{i=1}^d \sigma_i^2$ and
$$
	\mu_i = T_{0i} + T_{3i} + T_{4i} \qquad \mbox{and} \qquad \sigma_i^2 = T_{1i}^2 + T_{2i}^2 + 2T_{3i}^2 + 2T_{4i}^2 + T_{5i}^2.
$$
Hence, the expected acceptance rate satisfies
$$	
	\mathrm{E}[\alpha(x,y)] \rightarrow \Phi(\tsfrac{\mu}{\sigma}) + \mathrm{e}^{\mu + \sigma^2/2} \Phi(-\sigma - \tsfrac{\mu}{\sigma})
	\qquad \mbox{as $d \rightarrow \infty$}
$$
where $\Phi$ is the standard normal cumulative distribution function.
\end{theorem}

\begin{proof}
By Lemma \ref{lem2} it is sufficient to only consider \eqref{MH0} in the case where all matrices are diagonal matrices, e.g. $A = \operatorname{diag}(\lambda_1^2,\dotsc,\lambda_d^2)$, $\mathcal{A} = \operatorname{diag}(\tilde{\lambda}_1^2,\dotsc,\tilde{\lambda}_d^2)$, $M = \operatorname{diag}(M_1,\dotsc,M_d)$, $G = \operatorname{diag}(G_1,\dotsc,G_d)$, $m_i = \! \lambda_i^{-2} b_i$, and $\tilde{m}_i = \tilde{\lambda}_i^{-2} \beta_i$.  Then, in equilibrium we have 
$$
	x_i = m_i + \frac{1}{\lambda_i} \xi_i
$$
where $\xi_i \sim N(0,1)$ and using $\tilde{m}_i = G_i \tilde{m}_i + M_i^{-1}\beta_i$ we have
\begin{align*}
	y_i &= G_i x_i + \frac{\beta_i}{M_i} + \frac{(1-G_i^2)^{1/2}}{\tilde{\lambda}_i} \nu_i \\
	&= G_i \left( m_i + \frac{1}{\lambda_i} \xi_i \right) + (1-G_i) \tilde{m}_i  + \frac{g_i^{1/2}}{\tilde{\lambda}_i} \nu_i \\
	&= \tilde{m}_i + G_i \hat{r} + \frac{G_i}{\lambda_i} \xi_i + \frac{g_i^{1/2}}{\tilde{\lambda}_i} \nu_i 
\end{align*}
where $\nu_i \sim N(0,1)$.  From Lemma \ref{lem1} we also have $Z = \sum_{i=1}^d Z_i$ where
$$
	Z_i = -\tsfrac{1}{2} (\lambda_i^2 - \tilde{\lambda}_i^2)(y_i^2 - x_i^2) + (b_i-\beta_i)(y_i - x_i).
$$
Substituting $x_i$ and $y_i$ as above, using identities such as $\lambda_i^2 \tilde{\lambda}_i^{-2} = 1 + \tilde{r}_i$ and $(b_i - \beta_i)\lambda_i^{-2} = \hat{r}_i + r_i \tilde{m}_i$, then after some algebra we eventually find
\begin{align*}
	Z_i &= T_{0i} + T_{1i}\xi_i + T_{2i}\nu_i + T_{3i} \xi_i^2 + T_{4i} \nu_i^2 + T_{5i} \xi_i \nu_i.
\end{align*}
Hence
$$
	\mu_i:= \mathrm{E}[Z_i] = T_{0i} + T_{3i} + T_{4i}
$$
and
\begin{align*}
	\sigma_i^2 &:= \mathrm{Var}[Z_i] = \mathrm{E}[Z_i^2] - \mathrm{E}[Z_i]^2 \\
	&= \left( T_{0i}^2 + T_{1i}^2 + T_{2i}^2 + 3 T_{3i}^2 + 3 T_{4i}^2 + T_{5i}^2 
	+ 2 T_{0i} T_{3i} + 2 T_{0i} T_{4i} + 2 T_{3i} T_{4i} \right) \\
	& \qquad - \left( T_{0i} + T_{3i} + T_{4i} \right)^2 \\
	&= T_{1i}^2 + T_{2i}^2 + 2 T_{3i}^2 + 2 T_{4i}^2 + T_{5i}^2	 
\end{align*}
and
\begin{align*}
	Z_i - \mu_i &= T_{1i} \xi_i + T_{2i} \nu_i + T_{3i} (\xi_i^2 - 1) + T_{4i} (\nu_i^2 - 1) + T_{5i} \xi_i \nu_i.
\end{align*}
Therefore, for any $d \in \mathbb{N}$ and $\delta > 0$ we can bound the Lyapunov condition in Theorem \ref{thm clt} as follows 
\begin{align*}
	\frac{1}{s_d^{2+\delta}} \sum_{i=1}^d \mathrm{E}[|Q_i-\mu_i|^{2+\delta}] 
	&\leq \frac{5^{2+\delta}}{s_d^{2+\delta}} \sum_{j=1}^5 C_j(\delta) \sum_{i=1}^d |T_{ji}|^{2+\delta} \\
	&\leq 5^{2+\delta} \sum_{j=1}^5 C_j(\delta) \frac{\sum_{i=1}^d |T_{ji}|^{2+\delta}}{\left( \sum_{i=1}^d T_{ji}^2 \right)^{1+\delta/2} }
\end{align*}
where $C_1(\delta) = C_2(\delta) = \mathrm{E}[|\xi|^{2+\delta}]$ , $C_3(\delta) = C_4(\delta) = \mathrm{E}[|\xi^2-1|^{2+\delta}]$
and $C_5(\delta) = \mathrm{E}[|\xi|^{2+\delta}]^2$, and $\xi \sim N(0,1)$.

Therefore, if \eqref{eq Tcond} holds then the result follows from Theorem \ref{thm clt} and \eqref{eq uf1}.
\end{proof}

\section{Expected squared jump size for a general matrix splitting}
\label{sec jumpsize}

The efficiency of a MCMC method is usually given in terms of the integrated autocorrelation time which may be thought of as ``the number of dependent sample points from the Markov chain needed to give the variance reducing power of one independent point'', see e.g. \cite[\S 6.3]{N1993}.  Unfortunately, we are unable to directly estimate this quantity for our matrix splitting methods, and it depends on the statistic of concern.  As a proxy we instead consider the expected squared jump size of the Markov chain in a direction $q \in \bbR^d$,
$$
	\mathrm{E}[(q^T(x' - x))^2]
$$
where $x, x' \sim N(A^{-1}b,A^{-1})$ are successive elements of the Markov chain in equilibrium.  We will only consider the cases where $q$ is an eigenvector of the precision or covariance matrix.  It is related to the integrated autocorrelation time for the linear functional $q^T(\cdot)$ by
$$
	\mathrm{Corr}[q^T x,q^T x'] = 1 - \frac{\mathrm{E}[(q^T(x' - x))^2]}{2 \mathrm{Var}[q^T x]}
$$
so that large squared jump size implies small first-order autocorrelation.

This is similar to the approach used for analysing the efficiency of RWM, MALA and HMC, where the expected squared jump size of an arbitrary component of the Markov chain is considered (see e.g. \cite{BRS2009} and \cite{BPRSS2010}) .  

We will need the following technical lemma whose proof is in the Appendix.

\begin{lemma}
\label{lem 5}
Suppose that $\{ t_i \}_{i=1}^\infty$ is a sequence of real-valued numbers and $r > 0$.  Then, for any $k \in \mathbb{N}$,
\begin{equation}
\label{eq l1a}
	\lim_{d\rightarrow \infty} \frac{\sum_{i=1}^d |t_i|^r}{\left( \sum_{i=1}^d t_i^2 \right)^{r/2}} = 0 
	\quad \Rightarrow \quad
	\lim_{d\rightarrow \infty} \frac{\sum_{i=1, i \neq k}^d |t_i|^r}{\left( \sum_{i=1, i \neq k}^d t_i^2 \right)^{r/2}} = 0.
\end{equation}
\end{lemma}

The following main theorem for this section is a generalization of \cite[Prop. 3.8]{BPRSS2010}.

\begin{theorem}
\label{thm jumpsize}
Suppose that $M$ and $N$ in \eqref{MH0} are functions of $A$, and the Markov chain is in equilibrium, i.e. $x \sim N(A^{-1}b,A^{-1})$.
With the same definitions as in Theorem \ref{thm accept}, let $q_i$ be a normalized eigenvector of $A$ corresponding to $\lambda_i^2$ ($i^{th}$ column of $Q$).  If there exists a $\delta > 0$ such that \eqref{eq Tcond} is satisfied, then 
\begin{equation}
\label{eq jump1}
	\mathrm{E}[(q_i^T (x'-x))^2] \rightarrow U_1 U_2 + U_3 \qquad \mbox{as $d \rightarrow \infty$}
\end{equation}
where
\begin{align*}
	U_1 &= \tilde{g}_i^2 \hat{r}_i^2 + \frac{\tilde{g}_i^2}{\lambda_i^2} + \frac{g}{\tilde{\lambda}_i^2}, \\
	U_2 &= \Phi(\tsfrac{\mu^-}{\sigma^-}) + \mathrm{e}^{\mu^- + (\sigma^-)^2/2} \Phi(-\sigma^- - \tsfrac{\mu^-}{\sigma^-}), \\
	|U_3| &\leq ( \sigma_i^2 + \mu_i^2 )^{1/2}  
	\times \left( \tilde{g}_i^4 \hat{r}_i^4 + 3 \frac{\tilde{g}_i^4}{\lambda_i^4} + 3 \frac{g_i^2}{\tilde{\lambda}_i^4} + 6 \frac{\tilde{g}_i^4 \hat{r}_i^2}{\lambda_i^2} + 6 \frac{\tilde{g}_i^2 g_i \hat{r}^2}{\tilde{\lambda}_i^2} + 6 \frac{\tilde{g}_i^2 g}{\lambda_i^2 \tilde{\lambda}_i^2} \right)^{1/2}
\end{align*}
and where $\mu^- := \sum_{j=1,j\neq i}^d \mu_j$ and $(\sigma^-)^2 := \sum_{j=1,j\neq i}^d \sigma_j^2$.
\end{theorem}

\begin{proof}
Under the coordinate transformation $x \leftrightarrow Q^T x$, \eqref{MH0} becomes \eqref{MH1} and $\mathrm{E}[(q_i^T(x'-x))^2]$ becomes
$\mathrm{E}[(x_i' - x_i)^2]$.  Therefore it is sufficient to only consider the squared jump size of an arbitrary coordinate of the Markov chain for the case when all matrices are diagonal matrices.  As in the proof of Theorem \ref{thm accept}, let $A = \operatorname{diag}(\lambda_1^2,\dotsc,\lambda_d^2)$, $\mathcal{A} = \operatorname{diag}(\tilde{\lambda}_1^2,\dotsc,\tilde{\lambda}_d^2)$, $M = \operatorname{diag}(M_1,\dotsc,M_d)$, $G = \operatorname{diag}(G_1,\dotsc,G_d)$, $m_i = \lambda_i^{-2} b_i$, and $\tilde{m}_i = \tilde{\lambda}_i^{-2} \beta_i$.  Since the chain is in equilibrium we have $x_i = m_i + \lambda_i^{-1} \xi_i$ for $\xi_i \sim N(0,1)$ and $y_i = \tilde{m}_i + G_i \hat{r} + G_i \lambda_i^{-1} \xi_i + g_i^{1/2} \tilde{\lambda}_i^{-1} \nu_i$ where $\nu_i \sim N(0,1)$.  Now let $I^d$ be the indicator function such that
$$	
	I^d := \begin{cases}
		1 & \mbox{if $u < \alpha(x,y)$} \\
		0 & \mbox{otherwise}
	\end{cases}
$$
where $u \sim U[0,1]$.  Then $x' = I^d (y-x) + (1-I^d) x$ and
$
	(x'-x)^2 = I^d(y-x)^2.
$
For given coordinate $i$, define another indicator function $I^{d-}$ such that
$$
	I^{d-} := \begin{cases}
		1 & \mbox{if $u < \alpha^-(x,y)$} \\
		0 & \mbox{otherwise}
	\end{cases}
$$
where $\alpha^-(x,y) := 1 \wedge \exp( \sum_{j=1,j\neq i}^d Z_j )$, and define
$
	e_j := I^{d-}(y-x)^2.
$
The proof strategy is now to approximate $\mathrm{E}[(x_i' - x_i)^2]$ with $\mathrm{E}[e_i]$;
$$
	\mathrm{E}[(x_i'-x_i)^2] = \mathrm{E}[e_i] + \mathrm{E}[(x_i'-x_i)^2 - e_i].
$$
First, consider $\mathrm{E}[e_i]$;
\begin{align*}
	\mathrm{E}[e_i]  
	&= \mathrm{E}[I^{d-}(y_i-x_i)^2] \\
	&= \mathrm{E}[(y_i-x_i)^2] \mathrm{E}[\alpha^-(x,y)] \\
	&= \mathrm{E}\left[ \left(-\hat{r}_i - \frac{\tilde{g}_i}{\lambda_i}\xi_i + \frac{g^{1/2}}{\tilde{\lambda}_i} \nu_i \right)^2 \right] \mathrm{E}[\alpha^-(x,y)]\\
	&= U_1 \mathrm{E}[\alpha^-(x,y)].
\end{align*}
Also, by Theorem \ref{thm accept} (using Lemma \ref{lem 5} to ensure the appropriate condition for Theorem \ref{thm accept} is met) we obtain $\mathrm{E}[\alpha^-(x,y)] \rightarrow U_2$ as $d \rightarrow \infty$.

The error is bounded using the Cauchy-Schwarz inequality;
\begin{align*}
	|\mathrm{E}[(x_i'-x_i)^2 - e_i]|
	&= |\mathrm{E}[(I^d-I^{d-})(y_i-x_i)^2]| \\
	&\leq \mathrm{E}[(\alpha(x,y)-\alpha^-(x,y))^2]^{1/2} \mathrm{E}[(y_i-x_i)^4]^{1/2}.
\end{align*}	
Since $1\wedge \mathrm{e}^X$ is Lipschitz with constant $1$, and using results from the proof of Theorem \ref{thm accept}, we obtain
$$
	\mathrm{E}[(\alpha(x,y)-\alpha^-(x,y))^2]^{1/2} 
	\leq \mathrm{E}[Z_i^2]^{1/2} 
	= ( \sigma_i^2 + \mu_i^2 )^{1/2},
$$
and simple algebra yields
\begin{align*}
	\mathrm{E}[(y_i-x_i)^4]^{1/2}
	&= \mathrm{E}\left[ \left( -\hat{r}_i - \frac{\tilde{g}_i}{\lambda_i}\xi_i + \frac{g^{1/2}}{\tilde{\lambda}_i} \nu_i \right)^4 \right]^{1/2} \\
	&= \left( \tilde{g}_i^4 \hat{r}_i^4 + 3 \frac{\tilde{g}_i^4}{\lambda_i^4} + 3 \frac{g_i^2}{\tilde{\lambda}_i^4} + 6 \frac{\tilde{g}_i^4 \hat{r}_i^2}{\lambda_i^2} + 6 \frac{\tilde{g}_i^2 g_i \hat{r}^2}{\tilde{\lambda}_i^2} + 6 \frac{\tilde{g}_i^2 g}{\lambda_i^2 \tilde{\lambda}_i^2} \right)^{1/2}.
\end{align*}
\end{proof}

The terms in the theorem above are quite lengthy, but in many situations they simplify.  For example, we may have the situation where $\mu^- \rightarrow \mu$ and $\sigma^- \rightarrow \sigma$ as $d \rightarrow \infty$, so $U_2$ becomes the expected acceptance rate for the algorithm.  Also, it may be possible to derive a bound such as
$$
	|U_3| \leq C (r_i + \hat{r}_i)
$$	
so that $U_3$ is small if both the relative error of the $i^{\mathrm{th}}$ eigenvalue and error of the means are small.

\section{Examples}
\label{sec examples}

\subsection{Discretized Langevin diffusion - MALA}
\label{sec langevin}

The proposal for MALA is obtained from the Euler-Maruyama discretization of a Langevin diffusion process $\{ z_t \}$ which satisfies the stochastic differential equation
\[
	\frac{\dd z_t}{\dd t} = \frac{1}{2} \nabla \log \pi(z_t) + \frac{\dd W_t}{\dd t},
\]
where $W_t$ is standard Brownian motion in $\bbR^d$.  Since the diffusion process has the desired target distribution as equilibrium, one might expect a discretization of the diffusion process to almost preserve the desired target distribution.  Given target $N(A^{-1}b,A^{-1})$, current state $x \in \bbR^d$ and time step $h >0$, the MALA proposal $y \in \bbR^d$ is defined as
\begin{equation}
\label{MALA prop}
	y = (I - \tsfrac{h}{2} A) x + \tsfrac{h}{2} b + \sqrt{h} \xi 
\end{equation}
where $\xi \sim N(0,I)$.  Identifying this with the AR(1) process \eqref{eq ar1} and applying Corollary \ref{lem equiv} we have proved the following theorem.  

\begin{theorem}
\label{thm MALA split}
MALA corresponds to the matrix splitting
\begin{align*}
	M &= \tsfrac{2}{h}(I - \tsfrac{h}{4}A), & 
	\mathcal{A} &= (I - \tsfrac{h}{4}A) A, \\
	N &= \tsfrac{2}{h}(I - \tsfrac{h}{4}A)(I-\tsfrac{h}{2}A), &
	\beta &= (I - \tsfrac{h}{4}A) b.
\end{align*}
\end{theorem}

Thus, MALA corresponds to a matrix splitting where $M$ and $N$ are functions of $A$ and our theory applies.  An important feature of MALA is that $\hat{r}_i = 0$ for all $i$.  This greatly simplifies the results in Theorems \ref{thm accept} and \ref{thm jumpsize} and we recover \cite[Cor. 1]{BRS2009} and the simpler results in \cite[Thm. 7]{RR2001}.  The theorem below is a special case of Theorem \ref{thm genlang conv} so we omit the proof.

\begin{theorem}
\label{thm MALA conv}
If there exist positive constants $c$ and $C$, and $\kappa \geq 0$ such that the eigenvalues $\lambda_i^2$ of $A$ satisfy
$$
	c i^\kappa \leq \lambda_i \leq C i^\kappa
$$
and if $h = l^2 d^{-r}$ for some $l >0$ and $r = \tsfrac{1}{3} + 2 \kappa$ then MALA satisfies
$$
	\mathrm{E}[\alpha(x,y)] \rightarrow 2 \Phi\left( -\tsfrac{l^3 \sqrt{\tau}}{8} \right)
$$
and
\begin{equation}
\label{eq jump}
	\mathrm{E}[(x_i-x_i)^2] \rightarrow 2 l^2 d^{-r} \Phi\left( -\tsfrac{l^3 \sqrt{\tau}}{8} \right) + \mathrm{o}(d^{-r})
\end{equation}
as $d \rightarrow \infty$ where $\tau = \lim_{d \rightarrow \infty} \frac{1}{d^{1 + 6 \kappa}} \sum_{i=1}^d \lambda_i^6$.
\end{theorem}  

Thus, the performance of MALA depends on the choice of $h$ which is usually tuned (by tuning $l$) to maximise the expected jump distance.  From \eqref{eq jump}, using $s = l \tau^{1/6}/2$, we have 
$$
	\max_{l > 0} 2 l^2 d^{-r} \Phi\left( -\tsfrac{l^3 \sqrt{\tau}}{8} \right) = \max_{s > 0} \frac{8 d^{-r}}{\tau^{1/3}} s^2 \Phi(-s^3),
$$
which is maximised at $s_0=0.8252$, independent of $\tau$.  Therefore, the acceptance rate that maximises expected jump distance is $2 \Phi(-s_0^3) = 0.574$.  This result was first stated in \cite{RR1998}, then more generally in \cite{RR2001,BRS2009}.  In practice, $h$ (equivalently $l$) is adjusted so that the acceptance rate is approximately $0.574$ to maximise the expected jump size.  This acceptance rate is independent of $\tau$, so it is independent of the eigenvalues of $A$.  However, the efficiency of MALA still depends on $\tau$, and as $\tau$ increases the expected square jump distance will decrease by a factor $\tau^{1/3}$.  

The rationale for studying the case when $d \rightarrow \infty$ is that it is a good approximation for cases when $d$ is `large' and finite (see eg. \cite[Fig. 1]{RR1998} or \cite{RR2001}).  However, the results above suggest that we should take $h \rightarrow 0$ as $d \rightarrow \infty$, to achieve at best an expected jump size that also tends towards $0$ as $d \rightarrow \infty$.  The only good point about these results is that it demonstrates superior asymptotic performance of MALA over RWM (see \cite[Thms. 1-4 and Cor. 1]{BRS2009} and \cite[Fig. 1]{RR1998}).

To understand the convergence of MALA to equilibrium (burn in) we would like to know the ``spectral gap'' or second largest eigenvalue of the transition kernel, as this determines the rate of convergence.  As far as we are aware this is an open problem.

We can also use Theorem \ref{thm MALA split} to analyse the unadjusted Langevin algorithm (ULA) in \cite{RT1996}, which is simply the MALA proposal chain from \eqref{MALA prop} without the accept/reject step in the MH algorithm.  From Theorem \ref{thm MALA split} we see that it does not converge to the correct target distribution since $\mathcal{A} \neq A$.  Instead of converging to $N(A^{-1}b, A^{-1})$, ULA converges to $N(A^{-1}b,\mathcal{A}^{-1})$.  This is the reason why it is usually used as a proposal in the MH algorithm (MALA).  Indeed, the authors of \cite{RT1996} note that ULA has poor convergence properties.  We see here that it converges to the wrong target distribution, and from \cite{F2013,FP2015} we know its convergence rate to this wrong target distribution depends on the spectral radius of $G = I-\tsfrac{h}{2}A$, which is close to $1$ when $h$ is small.

Despite possibly having slow convergence per iteration, both MALA and ULA are cheap to compute.  Since $G = I-\tsfrac{h}{2}A$, we only require a single matrix-vector multiplication with $A$ at each iteration and since $\Sigma = h I$, i.i.d. sampling from $N(0,\Sigma)$ at each iteration is cheap.

\subsection{Discretized Langevin diffusion - more general algorithms}

It is possible to generalise MALA by `preconditioning' the Langevin diffusion process and using a discretization scheme that is not Euler-Maruyama.  For symmetric positive definite matrix $V \in \bbR^{d \times d}$ (the `preconditioner') consider a Langevin process $\{z_t\}$ satisfying the stochastic differential equation
\begin{equation}
\label{gen langevin}
	\frac{\dd z_t}{\dd t} = \frac{1}{2} V \nabla \log \pi(z_t) + \frac{\dd W_t}{\dd t}
\end{equation}
where $W_t$ is Brownian motion in $\bbR^d$ with covariance $V$.  For $\theta \in [0,1]$, time step $h >0$ and current state $x \in \bbR^d$ define a proposal $y \in \bbR^d$ by discretizing \eqref{gen langevin} as
$$
	y - x = \tsfrac{h}{2} V \nabla \log \pi(\theta y + (1-\theta)x) + \sqrt{h} \nu
$$
where $\nu \sim N(0,V)$.  Equivalently, when the target is $N(A^{-1}b,A^{-1})$,
\begin{equation}
\label{gen langevin prop}
	y = (I + \tsfrac{\theta h}{2} VA)^{-1} \left[ (I - \tsfrac{(1-\theta)h}{2} VA)x + \tsfrac{h}{2} V b + (h V)^{1/2} \xi \right]
\end{equation}
where $\xi \sim N(0,I)$.

Different choices of $\theta$ and $V$ give different AR(1) process proposals.  For example, MALA has $\theta = 0$ and $V = I$ and the preconditioned Crank-Nicolson algorithm (pCN) corresponds to $\theta = \tsfrac{1}{2}$ and $V = A^{-1}$.  Table \ref{tab1} describes several more examples.

\begin{table}
\begin{tabular}{|l|l|p{8.5cm}|}
\hline
$\theta = 0$ & $V = I$ & MALA and ULA, see e.g. \cite{RT1996}.  In \cite{BRS2009} the Simplified Langevin Algorithm (SLA) is equivalent to MALA for Gaussian target distributions. \\ \hline
$\theta = 0$ & $V = A^{-1}$ & Proposal used in \cite{PST2012}. After a change of variables $x \leftrightarrow Q^T V^{-1/2}x$ it is the Preconditioned Simplified Langevin Algorithm (P-SLA) as in \cite{BRS2009}. \\ \hline
$\theta \in [0,1]$ & $V = I$ & Proposal for the so-called $\theta$-SLA method in \cite{BRS2009}.  Called Crank-Nicolson method in \cite{CRSW2013} when $\theta = \tsfrac{1}{2}$. \\ \hline
$\theta = \tsfrac{1}{2}$ & $V = A^{-1}$ & Preconditioned Crank-Nicolson method (pCN), see e.g. \cite{CRSW2013}.  
\\ \hline
\end{tabular}
\caption{Different choices of $\theta$ and $V$ in \eqref{gen langevin prop} lead to different proposals for the MH algorithm.  Other choices are possible.}
\label{tab1}
\end{table}

Identifying \eqref{gen langevin prop} with \eqref{eq ar1} and applying Corollary \ref{lem equiv} we can prove the following theorem.

\begin{theorem}
\label{thm genlang}
The general Langevin proposal \eqref{gen langevin prop} corresponds to the matrix splitting
\begin{align*}
	M &= \tsfrac{2}{h} V^{-1/2} W (I + \tsfrac{\theta h}{2} B ) V^{-1/2}, &
	\mathcal{A} &= V^{-1/2} W B V^{-1/2} = \tilde{W} A, \\
	N &= \tsfrac{2}{h} V^{-1/2} W (I - \tsfrac{(1-\theta)h}{2} B) V^{-1/2}, &
	\beta &= V^{-1/2} W V^{1/2} b = \tilde{W} b,
\end{align*}
where $B = V^{1/2} A V^{1/2}$, $W = I + (\theta-\tsfrac{1}{2}) \tsfrac{h}{2} B$ and $\tilde{W} = I + (\theta-\tsfrac{1}{2}) \tsfrac{h}{2} AV$.
\end{theorem}

\begin{proof}
Take $G = (I + \frac{\theta h}{2}VA)^{-1}(I - \tsfrac{(1-\theta)h}{2} VA)$, $g = \tsfrac{h}{2}(I + \frac{\theta h}{2}VA)^{-1} V b$ and $\Sigma = (I + \frac{\theta h}{2}VA)^{-1} (hV) (I + \frac{\theta h}{2}AV)^{-1}$, then \eqref{gen langevin prop} is the same as \eqref{eq ar1}.  To apply Corollary \ref{lem equiv} we first check that $G \Sigma$ is symmetric.  We have
$$
	G = V^{1/2} (I + \tsfrac{\theta h}{2} B)^{-1} (I - \tsfrac{(1-\theta)h}{2}B) V^{-1/2}
	\quad \mbox{and} \quad
	\Sigma = h V^{1/2} (I + \tsfrac{\theta h}{2}B)^{-2} V^{1/2}
$$
so that 
\begin{align*}
	G \Sigma = h V^{1/2} (I + \tsfrac{\theta h}{2}B)^{-1} (I - \tsfrac{(1-\theta)h}{2}B) (I + \tsfrac{\theta h}{2}B)^{-2} V^{1/2}
\end{align*}
which is symmetric since $V$ and $B$ are symmetric.  Applying Corollary \ref{lem equiv} then yields the result.
\end{proof}

Therefore, the proposal target for \eqref{gen langevin prop} is $N(A^{-1}b, (\tilde{W}A)^{-1})$, and if $\theta \neq \tsfrac{1}{2}$, then $\tilde{W} \neq I$, $\mathcal{A} \neq A$, and the target and proposal target disagree.  

If $\theta = \tsfrac{1}{2}$, then the proposal target and target distributions are the same, the MH accept/reject step is redundant, and we can use \cite{F2013,FP2015,FP2014} to analyse and accelerate the AR(1) process \eqref{gen langevin prop}.  

To evaluate the performance of the MH algorithm with proposal \eqref{gen langevin prop} when $\mathcal{A} \neq A$ we would like to be able to apply Theorems \ref{thm accept} and \ref{thm jumpsize}, but we see in Theorem \ref{thm genlang} that the splitting matrices are functions of both $B$ and $V$ (not $A$) so we cannot directly apply our new theory.  A change of coordinates will fix this!  The following lemma is a result of simple algebra and Theorem \ref{thm genlang}.

\begin{lemma}
\label{lem ch1}
Under the change of coordinates 
$$
	x \leftrightarrow V^{-1/2} x
$$
the MH algorithm with target $N(A^{-1}b,A^{-1})$ and proposal \eqref{gen langevin prop} is transformed to the MH algorithm defined by
\begin{equation}
\label{eq hat1}
\begin{split}
	\mbox{Target:} & \qquad N(B^{-1} V^{1/2} b, B^{-1}), \\
	\mbox{Proposal:} & \qquad (I + \tsfrac{\theta h}{2} B) y = (I - \tsfrac{(1-\theta)h}{2} B) x + \tsfrac{h}{2} V^{1/2} b + h^{1/2} \xi, 
\end{split}
\end{equation}
where $\xi \sim N(0,I)$ and $B = V^{1/2} A V^{1/2}$.  Moreover, the proposal \eqref{eq hat1} corresponds to the matrix splitting $B = M-N$
\begin{align*}
	M &= \tsfrac{2}{h} W (I + \tsfrac{\theta h}{2} B ), &
	\mathcal{A} &= W B, \\
	N &= \tsfrac{2}{h} W (I - \tsfrac{(1-\theta)h}{2} B), &
	\beta &= W V^{1/2} b,
\end{align*}
where $W = I + (\theta-\tsfrac{1}{2}) \tsfrac{h}{2} B$.
\end{lemma}

Thus, we have transformed the MH algorithm with target $N(A^{-1}b,A^{-1})$ and proposal \eqref{gen langevin prop} to a MH algorithm where the splitting matrices are functions of the target precision matrix, and we can apply Theorems \ref{thm accept} and \ref{thm jumpsize} to \eqref{eq hat1} to find the expected acceptance rate and expected jump size of the MH algorithm with target $N(A^{-1}b,A^{-1})$ and proposal \eqref{gen langevin prop}.

\begin{theorem}
\label{thm genlang conv}
Suppose there are constants $c,C > 0$ and $\kappa \geq 0$ such that the eigenvalues $\lambda_i^2$ of $V A$ (equivalently, $V^{1/2}AV^{1/2}$) satisfy
\[
	c i^\kappa \leq \lambda_i \leq C i^\kappa \qquad \mbox{for $i=1,\dotsc,d$}.
\]
If $h = l^2 d^{-r}$ for $r=\tsfrac{1}{3} + 2 \kappa$ and $l > 0$, and $\tau = \lim_{d\rightarrow \infty} \frac{1}{d^{6\kappa + 1}} \sum_{i=1}^d \lambda_i^6$ then a MH algorithm with proposal \eqref{gen langevin prop} and target $N(A^{-1}b,A^{-1})$ satisfies
\begin{equation}
\label{eq expect}
	\mathrm{E}[\alpha(x,y)] \rightarrow 2 \Phi\left( -\frac{l^3 |\theta-\tsfrac{1}{2}| \sqrt{\tau}}{4} \right) 
\end{equation}
and for normalised eigenvector $q_i$ of $V^{1/2}A V^{1/2}$ corresponding to $\lambda_i^2$,
\begin{equation}
\label{eq jump2}
	\mathrm{E}[ |q_i^T V^{-1/2} (x' - x)|^2 ] \rightarrow 2 l^2 d^{-r} \Phi\left( -\frac{l^3 |\theta-\tsfrac{1}{2}| \sqrt{\tau}}{2} \right) + \mathrm{o}(d^{-r})
\end{equation}
as $d \rightarrow \infty$.
\end{theorem}

The proof of Theorem \ref{thm genlang conv} is in the Appendix.

We see that $V$ plays the same role of preconditioning as for solving a linear system and we should choose it minimise the condition number  of $VA$, so that $\tau$ is minimised.

Although MALA and more general discretizations of Langevin diffusion have been successfully analyzed in \cite{BRS2009}, all of these results are stated for product distributions and `small' perturbations of product distributions.  Theorem \ref{thm genlang conv} extends their theory (in particular \cite[Cor. 1]{BRS2009}) to the Gaussian case with non-diagonal covariance and $\theta \in [0,1]$.  See also \cite[Thm. 7]{RR2001}.

For efficiency, as well as considering the expected squared jump size, we must also consider the computing cost of the proposal \eqref{gen langevin prop}, which requires the action of $(I + \frac{\theta h}{2}VA)^{-1}$ and an independent sample from $N(0,V)$, as well as the actions of $V$ and $A$.

MALA, with $\theta = 0$ and $V = I$, is very cheap to compute because we only have to invert the identity matrix and sample from $N(0,I)$ at each iteration (as well as multiply by $A$).  

Alternatively, pCN, with $\theta = \tsfrac{1}{2}$ and $V = A^{-1}$, requires us to sample from $N(0,A^{-1})$ which we have assumed is infeasible because we cannot factorize $A$.

\subsection{Hybrid Monte Carlo}

Another type of AR(1) process proposal that fits our theory are proposals from the Hybrid (or Hamiltonian) Monte Carlo algorithm (HMC), see e.g. \cite{DKPR1987,BPRSS2010,N1993}.    

HMC treats the current state $x \in \bbR^d$ as the initial position of a particle, the initial momentum $p \in \bbR^d$ of the particle is chosen independently at random, and then the motion of the particle is evolved according to a Hamiltonian system for a fixed amount of time.  The final position of the particle is the proposal.  Instead of solving the Hamiltonian system exactly, the evolution of the particle is approximated using a reversible, symplectic numerical integrator.  For example, the leap-frog method (also called the Stormer-Verlet method) is an integrator that preserves a modified Hamiltonian, see e.g. \cite{HLW}.  Hence the proposal $y$, for target $N(A^{-1}b,A^{-1})$, is computed as follows; let $V \in \bbR^{d \times d}$ be a symmetric positive definite matrix and define a Hamiltonian function $H : \bbR^d \times \bbR^d \rightarrow \bbR$ by
\begin{equation}
\label{eq ham1}
	H(q,p) := \frac{1}{2} p^T V p + \frac{1}{2} q^T A q - b^T q. 
\end{equation}
Given a time step $h>0$, a number of steps $L \in \mathbb{N}$, and current state $x \in \bbR^d$, define $q_0 := x$, and sample $p_0 \sim N(0,V^{-1})$.  Then for $l=0,\dotsc,L-1$ compute
\begin{align*}
	p_{l+1/2} &= p_{l} - \tsfrac{h}{2} (A q_{l} - b), \\
	q_{l+1} &= q_{l} + h V p_{l+1/2}, \\
	p_{l+1} &= p_{l+1/2} - \tsfrac{h}{2} (A q_{l+1} - b).
\end{align*}
The proposal is then defined as $y := q_L$.  In matrix form we have
$$
	\left[ \twobyone{q_{l+1}}{p_{l+1}} \right] = K 	\left[ \twobyone{q_l}{p_l} \right] + J \left[ \twobyone{0}{\tsfrac{h}{2} b} \right]
$$
where $K,J \in \bbR^{2d \times 2d}$ are defined as
$$
	K = 
	\left[ \!\! \twobytwo{I}{0}{-\tsfrac{h}{2} A}{I} \right]
	\left[ \twobytwo{I}{h V}{0}{I}  \right]
	\left[ \!\! \twobytwo{I}{0}{-\tsfrac{h}{2} A}{I} \right]
	=\left[ \!\! \twobytwo{I-\tsfrac{h^2}{2}VA}{hV}{-h A + \tsfrac{h^3}{4}AVA}{I-\tsfrac{h^2}{2}AV} \!\right]
$$
and
$$
	J = 	
	\left[ \twobytwo{I}{0}{0}{I} \right]	
	+	
	\left[ \twobytwo{I}{0}{-\tsfrac{h}{2} A}{I} \right]
	\left[ \twobytwo{I}{h V}{0}{I} \right]
	=\left[ \twobytwo{2I}{hV}{-\tsfrac{h}{2}A}{2I - \tsfrac{h^2}{2}AV} \right].
$$
Hence, $y$ is given by 
\begin{equation}
\label{hmc prop2}
	\left[ \twobyone{y}{p_L} \right]
	=
	K^L \left[ \twobyone{x}{\xi} \right] + \sum_{l=0}^{L-1} K^l J \left[ \twobyone{0}{\tsfrac{h}{2}b} \right] 
	\quad\mbox{where $\xi \sim N(0,V)$},
\end{equation}
or equivalently, 
\begin{equation}
\label{hmc prop}
	y = (K^L)_{11} x + \left( SJ \left[ \twobyone{0}{\tsfrac{h}{2}b} \right] \right)_1 + (K^L)_{12} \xi
\end{equation}
where $\xi \sim N(0,V^{-1})$, $(K^L)_{ij}$ is the $ij$ block (of size $d\times d$) of $K^L$, $S = (I-K)^{-1} (I-K^L)$ and $(\cdot)_1$ are the first $d$ entries of the vector $(\cdot)$.

In the case of only one time step of the leap-frog integrator ($L=1$) then HMC is MALA \cite{BPRSS2010}.  Hence, we immediately know that the HMC proposal with $L=1$ is an AR(1) process where the proposal target and target distributions are not the same, and the expected acceptance rate and jump size are given by Theorem \ref{thm MALA conv}.  The case for $L>1$ is more complicated, but \eqref{hmc prop} is still an AR(1) process that can be expressed as a matrix splitting using \eqref{eq ar1b}.  The proofs of the following two results are in the Appendix.

\begin{theorem}
\label{thm HMC1}
The HMC proposal \eqref{hmc prop} corresponds to the matrix splitting 
\begin{align*}
	M &= \Sigma^{-1} (I+(K^L)_{11}), &
	\mathcal{A} &= \Sigma^{-1} (I - (K^L)_{11}^2), \\
	N &= \Sigma^{-1} (I+(K^L)_{11}) (K^L)_{11}, &
	\beta &= \Sigma^{-1} (I+(K^L)_{11}) \left(( SJ \left[ \twobyone{0}{\tsfrac{h}{2}b} \right] \right)_1, 
\end{align*}	
where $\Sigma = (K^L)_{12} V^{-1} (K^L)_{12}^T$.
\end{theorem}

\begin{corollary}
\label{cor 15}
The matrix splitting from HMC satisfies $\mathcal{A}^{-1} \beta \! = A^{-1} b$.
\end{corollary}

These results imply that the proposal target distribution for HMC is $N(A^{-1}b, \mathcal{A}^{-1})$ rather than the desired target $N(A^{-1}b,A^{-1})$, hence why this AR(1) process is used as a proposal in the MH algorithm.  

For the analysis of HMC we require the eigenvalues of the iteration matrix.  A proof of the following result is in the Appendix.

\begin{theorem}
\label{thm uhmc conv}
Let $\lambda_i^2$ be eigenvalues of $VA$.  Then the iteration matrix $G = (K^L)_{11}$ for the HMC proposal has eigenvalues 
$$
	G_i = \cos( L \theta_i )
$$
where $\theta_i = -\cos^{-1} ( 1 - \tsfrac{h^2}{2} \lambda_i^2 )$.
\end{theorem}

From this theorem we see how the eigenvalues of the iteration matrix depend on $V$, the number of time steps $L$, and the time step $h$.  Again we refer to $V$ as a preconditioner (as in \cite{BPSSS2011}) because it plays the same role as a preconditioner for solving systems of linear equations.  Alternatively, $V$ may be referred to as a mass matrix since $p$ in the Hamiltonian \eqref{eq ham1} is momentum and $H$ is energy.

To complete our analysis of HMC we restrict our attention to the case when $d \rightarrow \infty$ and try to apply Theorems \ref{thm accept} and \ref{thm jumpsize}.  These theorems require that the splitting matrices are functions of the target precision matrix.  A simple change of coordinates achieves this.

\begin{theorem}
\label{thm iso hmc}
Under the change of coordinates 
$$
	\left[ \twobyone{x}{p} \right] \leftrightarrow \mathcal{V}^{-1} \left[ \twobyone{x}{p} \right], \qquad \mbox{where } 
	\mathcal{V} = \left[ \twobytwo{V^{1/2}}{0}{0}{V^{-1/2}} \right] \in \bbR^{2d \times 2d},
$$
the Hamiltonian \eqref{eq ham1} and HMC with target $N(A^{-1}b,A^{-1})$ and proposal \eqref{hmc prop} are transformed to a Hamiltonian, and HMC defined by
\begin{equation}
\label{eq hmchat1}
\begin{split}
	\mbox{Hamiltonian:} & \qquad \mathcal{H}(x,p):= \tsfrac{1}{2} p^T p + \tsfrac{1}{2} x^T B x - (V^{1/2}b)^T x, \\
	\mbox{Target:} & \qquad N(B^{-1} V^{1/2} b, B^{-1}), \\
	\mbox{Proposal:} & \qquad 
	y = (\mathcal{K}^L)_{11} x + \left( \mathcal{S} \mathcal{J} \left[ \twobyone{0}{\tsfrac{h}{2} V^{1/2} b} \right] \right)_1 + (\mathcal{K}^L)_{12} \xi
\end{split}
\end{equation}
where $\xi \sim N(0,I)$, $B = V^{1/2} A V^{1/2}$, $\mathcal{S} = (I-\mathcal{K})^{-1}(I-\mathcal{K}^L)$,
$$
	\mathcal{K} = \left[ \twobytwo{I-\tsfrac{h^2}{2}B}{hI}{-h B + \tsfrac{h^3}{4}B^2}{I-\tsfrac{h^2}{2}B} \right]
	\qquad \mbox{and} \qquad
	\mathcal{J} = \left[ \twobytwo{2I}{hI}{-\tsfrac{h}{2}B}{2I - \tsfrac{h^2}{2}B} \right].
$$
Moreover, the proposal \eqref{eq hmchat1} corresponds to the matrix splitting $\mathcal{A} = M-N$
\begin{align*}
	M \!&=\! (\mathcal{K}^L)_{12}^{-2} (I\!+\!(\mathcal{K}^L)_{11}),\!\! &
	\mathcal{A} \!&=\! (\mathcal{K}^L)_{12}^{-2} (I \!-\! (\mathcal{K}^L)_{11}^2), \\
	N \!&=\! (\mathcal{K}^L)_{12}^{-2} (I\!+\!(\mathcal{K}^L)_{11}) (\mathcal{K}^L)_{11},\!\! &
	\beta \! &=\! (\mathcal{K}^L)_{12}^{-2} (I\!+\!(\mathcal{K}^L)_{11}) \left( \!\mathcal{S} \mathcal{J} \!\left[\! \twobyone{0}{\tsfrac{h}{2}V^{1/2} b} \!\right] \right)_1 \!. 
\end{align*}	
\end{theorem}

\begin{proof}
Use $K = \mathcal{V} \mathcal{K} \mathcal{V}^{-1}$ and $J = \mathcal{V} \mathcal{J} \mathcal{V}^{-1}$.
\end{proof}

Similar coordinate transformations are used in classical mechanics \cite[p. 103]{A1989}, see also \cite{BCSS2014}.

We now have splitting matrices that are functions of the target precision matrix, so we can apply Theorems \ref{thm accept} and \ref{thm jumpsize} to the HMC defined by \eqref{eq hmchat1} to reveal information about the performance of the original HMC algorithm.  To avoid being overly technical we restrict ourselves to the case where $\lambda_i = i^\kappa$ for some $\kappa \geq 0$.  This is still an extension to the results in \cite{BPSSS2011} since they only consider the case when $\kappa = 0$.  A proof is in the Appendix.

\begin{theorem}
\label{thm hmc conv}
Suppose that for some $\kappa \geq 0$, the eigenvalues of $V A$ (equivalently, $V^{1/2}AV^{1/2}$) satisfy
$$
	\lambda_i = i^\kappa \qquad \mbox{for $i=1,\dotsc,d$}.
$$
If $h = l d^{-r}$ for $r=\tsfrac{1}{4} + \kappa$ and $l > 0$, and $L = \lfloor \tsfrac{T}{h} \rfloor$ for fixed $T$, then the HMC algorithm (with proposal \eqref{hmc prop} and target $N(A^{-1}b,A^{-1})$) satisfies
\begin{equation}
\label{eq hmcexpect}
	\mathrm{E}[\alpha(x,y)] \rightarrow 2 \Phi\left( -\frac{l^2}{8\sqrt{2}\sqrt{1+4\kappa}} \right) 
\end{equation}
and for eigenvector $q_i$ of $V^{1/2}A V^{1/2}$ corresponding to $\lambda_i^2$,
\begin{equation}
\label{eq hmcjump}
	\mathrm{E}[ |q_i^T V^{-1/2} (x' - x)|^2 ] \rightarrow 4 \frac{1-\cos(\lambda_i T')}{\lambda_i^2} \Phi\left( -\frac{l^2 }{8 \sqrt{2}\sqrt{1+4\kappa}} \right) + \mathrm{o}(d^{-1/2})
\end{equation}
as $d \rightarrow \infty$, where $T' = Lh$.
\end{theorem}

The computational cost of HMC should also be considered.  Each proposal of HMC requires an independent sample from $N(0,V^{-1})$, $L$ matrix-vector products with $V$ and $L+1$ matrix-vector products with $A$.  Again, there is a balance to be struct optimizing the convergence rate relative to compute time (rather than iteration count).  Extending the results in \cite{BPRSS2010}, we have shown that in the case when the eigenvalues of $VA$ are $i^\kappa$, $T$ is fixed and $d \rightarrow \infty$, if we take $h = l d^{-1/4-\kappa}$ and $L = \lfloor \tsfrac{T}{h} \rfloor$ then the acceptance rate is $\mathcal{O}(1)$, so in high dimensions the number of iterations of HMC per independent sample should scale like $\mathcal{O}(d^{1/4+\kappa})$, which is an improvement over MALA which requires $\mathcal{O}(d^{1/3+\kappa})$.

This theory is for the leap-frog numerical integrator applied to the Hamiltonian system.  Higher order integrators are also suggested in \cite{BPRSS2010} and alternative numerical integrators based on splitting methods (in the ODEs context) are suggested in \cite{BCSS2014} that minimize the Hamiltonian error after $L$ steps of the integrator.  It may be possible to evaluate these other methods using Theorems \ref{thm accept} and \ref{thm jumpsize} after a change of variables.

We also note that the variant of HMC in \cite{BPSSS2011} corresponds to $V = A^{-1}$ and the change of variables $p \leftrightarrow V p$.  Since this method requires the spectral decomposition of $A$ for computing samples of $N(0,A^{-1})$ it is infeasible in our context.

\section{Concluding remarks}
\label{sec conclusion}

Designing proposals for the MH algorithm to achieve efficient MCMC methods is a challenge, particularly for non-Gaussian target distributions, and the job is made harder by the difficultly we have in analysing the convergence properties of MH algorithms.  By focusing on AR(1) process proposals in high dimension we have proven new theoretical results that provide us with criteria for evaluating AR(1) process proposals and guide us in constructing proposals for efficient MH algorithms.  

We have shown there is flexibility in how AR(1) processes are expressed, and by using the matrix splitting formalism we easily identify the proposal target distribution.  In particular, we have shown that all convergent AR(1) processes can be written in matrix splitting form, so that \eqref{eq ar1} and \eqref{eq ar1b} are interchangeable.  These include the proposals for MALA, HMC, and other discretized Langevin diffusion and discretized Hamiltonian dynamics proposals.  Since all AR(1) processes of the form \eqref{eq ar1b} correspond to fixed-scan Gibbs samplers, we conclude that MALA and HMC are fixed-scan Gibbs samplers, albeit with a proposal target distribution that is not the desired target distribution.

A special case is when the target distribution is normal (but not the same as the proposal target) and the splitting matrices are functions of the precision matrix for the target distribution.  In this case we proved new results for the expected acceptance rate and expected squared jump size when $d \rightarrow \infty$.  We showed how these quantities depend on the eigenvalues of the iteration matrix, and the difference between the proposal target distribution and the target distribution.  

Although our new results are for the special case case of normal target distributions, they keep the essential feature of non-normal targets, because the MH accept/reject step must be used to get convergence to the correct target distribution.  Our results also provide us with guidance for the case when the target distribution is non-normal, since we can take $N(\mathcal{A}^{-1}\beta,\mathcal{A})$ to be a local normal approximation to $\pi$.  Moreover, our assumption that the splitting matrices are functions of the target precision matrix $A$ (when the target is normal) is natural in high dimensions since it may be infeasible to factorize $A$.

Designing an efficient MH algorithm with an AR(1) process proposal is often a balancing act between minimising the integrated autocorrelation time (we use maximising expected jump size as a proxy for this) and minimising compute time for each iteration of the chain.  If the proposal target and target distributions are Gaussian and identical then it follows from the theory in \cite{F2013,FP2015,FP2014} that to construct an efficient AR(1) process we should try to satisfy the following conditions:
\begin{enumerate}
\item The spectral radius of $G$ should be as small as possible.
\item The action of $G$ or $M^{-1}$ should be cheap to compute.
\item Independent sampling from $N(g,\Sigma)$ or $N(M^{-1}b,M^{-1}(M^T+N)M^{-T})$ should be cheap to compute.
\end{enumerate}
However, if the Gaussian proposal target and target distributions are not the same, then our new theory suggests that in addition, we also require:
\begin{enumerate}
\setcounter{enumi}{3}
\item the difference between the desired target and proposal target distributions should be as small as possible in the sense that the difference in means should be small, and the relative difference in precision matrix eigenvalues should be small.
\end{enumerate}

In particular examples we can quantify these conditions using our theory.  For example, for proposals based on discretized generalised Langevin diffusion, Theorem \ref{thm genlang conv} shows us how the choice of symmetric positive definite matrix $V$ effects efficiency as it effects squared jump size in four ways.  Whilst choosing $V$ to maximise the limit in \eqref{eq jump2} (by minimising $r$ and $\tau$) we should balance this against the scaling and direction that $V$ induces on $E[q_i^T V^{-1/2}(x'-x)^2]$ through $q_i$ and $V^{-1/2}$ on the left-hand side of \eqref{eq jump2}.    

Another example is pCN which satisfies conditions $1$, $2$ and $4$ above, but not condition $3$.  In particular, $G$ is the diagonal matrix with entries all $\tsfrac{1-h/4}{1+h/4}$ on the diagonal, and $\mathcal{A} = A$ and $\beta = b$ so the proposal target and desired target are the same.  However, each iteration of pCN requires an independent sample from $N(0,A^{-1})$, which may be infeasible in high dimension.  In the special case when $A$ is diagonal, or a spectral decomposition of $A$ is available, then pCN satisfies all of our conditions for an efficient method.

By applying our new results to existing MCMC methods such as MALA, HMC, and other discretized Langevin diffusion and discretized Hamiltonian dynamics proposals we extended results already in the literature.  In particular, we extended results for Langevin proposals in \cite{BRS2009} to Gaussian targets with non-diagonal covariance and $\theta \in [0,1]$.  For HMC, we extended results in \cite{BPSSS2011} to the case when $\kappa \geq 0$ from $\kappa = 0$.  We have also derived a new formula for the eigenvalues of the iteration matrix of the HMC proposal.  

Proposals for MALA and HMC are examples of proposals that are constructed by discretizing a stochastic differential equation that preserves the target distribution.  Our theory allows us to broaden the class of AR(1) process proposals for the MH algorithm to more general AR(1) process proposals.

Whilst we do not specify any new acceleration strategies, our results are an important step in this direction because we give a criteria to evaluate AR(1) process proposals, including accelerations used in Fox and Parker \cite{F2013,FP2015,FP2014}.  Acceleration techniques for MH algorithms is an avenue for further research.

\appendix
\section{Proofs}

\subsection{Proof of Lemma \ref{lem1}}

First note that
$$
	q(x,y) 
	\propto 
	\exp( \tsfrac{1}{2} (My - Nx - \beta)^T (M+N)^{-1} (My-Nx - \beta)).
$$
Simple algebra then yields
\begin{align*}
	2 \log \left( \frac{\pi(y)q(y,x)}{\pi(x)q(x,y)} \right) \!\! &= -y^T A y + x^T A x + 2b^T(y-x)\\
	&\qquad - (Mx-Ny-\beta)^T (M+N)^{-1} (Mx-Ny-\beta) \\
	&\qquad + (My-Nx-\beta)^T (M+N)^{-1} (My-Nx-\beta) \\
	&= -y^T A y + x^T A x + 2b^T(y-x)\\
	& \qquad - ((M-N)(x+y))^T (M+N)^{-1} ((M+N)(x - y)) \\
	& \qquad + 2 \beta^T(M+N)^{-1} ((M+N)(x-y)) \\
	&= -y^T (A-\mathcal{A}) y + x^T (A-\mathcal{A}) x + 2(b-\beta)^T(y-x).
\end{align*}	

\subsection{Proof of Lemma \ref{lem 5}}

Suppose $\lim_{d\rightarrow \infty} (\sum_{i=1}^d |t_i|^r)/ (\sum_{i=1}^d t_i^2 )^{r/2}\\ = 0$.  Then for any $\epsilon > 0$ there exists a $D \in \mathbb{N}$ such that for any $d > D$,
$
	\sum_{i=1}^d |t_i|^r < \epsilon ( \sum_{i=1}^d t_i^2 )^{r/2}.
$
Then for any $k \in \mathbb{N}$, taking $\epsilon = 2^{-r/2}$, there exists a $D \geq k$ such that for any $d > D$,
$$
	|t_k|^r \leq \sum_{i=1}^d |t_i|^r < \frac{1}{2^{r/2}} \left( \sum_{i=1}^d t_i^2 \right)^{r/2}.
$$
Therefore, for any $d > D$, $t_k^2 < \tsfrac{1}{2} \sum_{i=1}^d t_i^2$ and so
$$
	\frac{\sum_{i=1, i \neq k}^d |t_i|^r}{\left( \sum_{i=1, i \neq k}^d t_i^2 \right)^{r/2}}  
	< \frac{ \sum_{i=1}^d |t_i|^r }{\left(\tsfrac{1}{2} \sum_{i=1}^d t_i^2 \right)^{r/2}} 
	= 2^{r/2} \frac{ \sum_{i=1}^d |t_i|^r }{\left(\sum_{i=1}^d t_i^2 \right)^{r/2}}.
$$

\subsection{Proof of Theorem \ref{thm genlang conv}}

We use the following technical lemma in the proof of Theorem \ref{thm genlang conv}.
\begin{lemma} \label{lem 6a} 
Suppose $\{t_i\}$ is a sequence of real numbers such that $0 < t_i \leq C d^{-1/3} (\tsfrac{i}{d})^{2\kappa}$ for $C>0$ and $\kappa \geq 0$.  If $s > 3$, then $
\lim_{d\rightarrow \infty} \sum_{i=1}^d t_i^s = 0
$.
\end{lemma}

\begin{proof}
\begin{align*}
	\lim_{d \rightarrow \infty} \sum_{i=1}^d t_i^s 
	\leq C^s \lim_{d \rightarrow \infty}  d^{1-s/3} \sum_{i=1}^d \tsfrac{1}{d} \left( \tsfrac{i}{d} \right)^{2\kappa s} 
	= C^s \lim_{d \rightarrow \infty}  d^{1-s/3} \int_0^1 z^{2\kappa s} \mathrm{d}z 
	= 0.
\end{align*}
\end{proof}

\emph{Proof of Theorem \ref{thm genlang conv}}.  First note that $VA$ and $V^{1/2}AV^{1/2}$ are similar, so they have the same eigenvalues.  Lemma \ref{lem ch1} implies that it is equivalent to study the MH algorithm with target and proposal given by \eqref{eq hat1}, and since the splitting matrices for \eqref{eq hat1} are functions of $B = V^{1/2}A V^{1/2}$ we can apply Theorems \ref{thm accept} and \ref{thm jumpsize} to \eqref{eq hat1}.   If we let $t_i = h \lambda_i^2$ and $\rho = (\theta-\tsfrac{1}{2})/2$ we have
\begin{align*}
	\tilde{\lambda}_i^2 &= (1+\rho t_i) \lambda_i^2, & G_i &= 1 - \frac{\frac{1}{2}t_i}{1 + \frac{\theta}{2}t_i}, &
	\tilde{g}_i &= \frac{\frac{1}{2}t_i}{1 + \frac{\theta}{2}t_i}, & g_i &= \frac{t_i (1 + \rho t_i)}{1 + \frac{\theta}{2}t_i}, \\
	r_i &= -\rho t_i, & 1 + \tilde{r}_i &= \frac{1}{1 + \rho t_i}, & \hat{r}_i &= 0,
\end{align*}
so that
\begin{align*}
	T_{0i} &= T_{1i} = T_{2i} = 0, \\
	T_{3i} &= \frac{-\frac{1}{2} \rho t_i^2 (1 + \rho t_i)}{(1 + \frac{\theta}{2}t_i)^2}, &
	T_{4i} &= \frac{\frac{1}{2} \rho t_i^2 }{(1 + \frac{\theta}{2}t_i)^2}, &
	T_{5i} &= \frac{\frac{1}{2} \rho t_i^{3/2} (1-\frac{1-\theta}{2} t_i) }{(1 + \frac{\theta}{2}t_i)^2}.
\end{align*}
To apply Theorems \ref{thm accept} and \ref{thm jumpsize} we first need to check \eqref{eq Tcond}.  For sufficiently large $d$ we have $1 \leq 1 + \frac{\theta}{2} t_i \leq \frac{3}{2}$ and $1 \leq 1 + \rho t_i \leq \frac{3}{2}$.  Then for some $\delta > 0$ we have
\begin{align*}
	\lim_{d \rightarrow \infty} \!\frac{\sum_{i=1}^d |T_{3i}|^{2+\delta}}{\left( \sum_{i=1}^d |T_{3i}|^2 \right)^{1 + \delta/2}}
	&\leq \lim_{d \rightarrow \infty} \left( \tsfrac{3}{2} \right)^{6+3\delta}  \frac{ \sum_{i=1}^d t_i^{4+2\delta}}{\left(  \sum_{i=1}^d t_i^4 \right)^{1 + \delta/2}} \\
	&\leq \lim_{d \rightarrow \infty} \left( \tsfrac{3}{2} \right)^{6+3\delta} \!(\tsfrac{C}{c})^{8+4\delta} d^{-\delta/2}  \frac{ \sum_{i=1}^d d^{-1} (\tsfrac{i}{d})^{8\kappa+4\delta\kappa}}{\left( \sum_{i=1}^d d^{-1} (\tsfrac{i}{d})^{8\kappa} \right)^{1 + \delta/2}} \\
	&= \lim_{d\rightarrow \infty} \left( \tsfrac{3}{2} \right)^{6+3\delta} \!(\tsfrac{C}{c})^{8+4\delta} d^{-\delta/2} \frac{ \int_0^1 z^{8\kappa+4\delta\kappa} \mathrm{d}z}{\left( \int_0^1 z^{8\kappa} \mathrm{d}z \right)^{1 + \delta/2}}    
	= 0.
\end{align*}
Similarly, we can check that \eqref{eq Tcond} is satisfied for $j=4,5$.  Now we can apply Theorem \ref{thm accept}, with
\begin{align*}
	\mu 
	&= \lim_{d \rightarrow \infty} \sum_{i=1}^d \mu_i 
	= \lim_{d \rightarrow \infty} \sum_{i=1}^d T_{3i} + T_{4i} \\
	&= \lim_{d \rightarrow \infty} -\tsfrac{\rho^2}{2} \sum_{i=1}^d \frac{t_i^3}{(1 + \frac{\theta}{2}t_i)^2} 
	= \lim_{d \rightarrow \infty} -\tsfrac{\rho^2}{2} \sum_{i=1}^d t_i^3 
	= - \frac{l^6 (\theta-\frac{1}{2})^2 \tau}{8},
\end{align*}
and using Lemma \ref{lem 6a},
\begin{align*}
	\sigma^2
	&= \lim_{d \rightarrow \infty} \sum_{i=1}^d \sigma_i^2 
	= \lim_{d \rightarrow \infty} \sum_{i=1}^d 2 T_{3i}^2 + 2 T_{4i}^2 + T_{5i} \\
	&= \lim_{d \rightarrow \infty} \sum_{i=1}^d \frac{1}{(1+\frac{\theta}{2}t_i)^4} \left( \tsfrac{1}{2}\rho^2 t_i^4 (1+\rho t_i)^2 + \tsfrac{1}{2} \rho^2 t_i^4 + \tsfrac{1}{4} \rho^2 t_i^3 (1-\tsfrac{1-\theta}{2}t_i)^2 \right)  \\
	&= \lim_{d \rightarrow \infty} \sum_{i=1}^d \left( \tsfrac{1}{2}\rho^2 t_i^4 (1+\rho t_i)^2 + \tsfrac{1}{2} \rho^2 t_i^4 + \tsfrac{1}{4} \rho^2 t_i^3 (1-\tsfrac{1-\theta}{2}t_i)^2 \right) \\
	&= \lim_{d \rightarrow \infty} \sum_{i=1}^d \tsfrac{1}{4} \rho^2 t_i^3 \\
	&= \lim_{d \rightarrow \infty} \frac{l^6 (\theta - \tsfrac{1}{2})^2 \tau}{4}.
\end{align*}
It follows that $\frac{\mu}{\sigma} = -\sigma - \frac{\mu}{\sigma} = -l^3 |\theta-\frac{1}{2}| \sqrt{\tau}/4$ and $\mu + \frac{\sigma^2}{2} = 0$.  Hence we obtain \eqref{eq expect}.

For the expected jump size, first note that
$$
	U_1 = \frac{\tilde{g}_i^2}{\lambda_i^2} + \frac{g}{\tilde{\lambda}_i^2} = \frac{h}{(1+\frac{\theta}{2}t_i)} \left( 1 + \frac{\frac{1}{4}t_i}{(1+\frac{\theta}{2}t_i)} \right) = l^2 d^{-r} + \mathrm{o}(d^{-1/3-r}). 
$$
Also, it is easy to show that $\mu_i = \mathrm{o}(d^{-1})$ and $\sigma_i^2 = \mathrm{o}(d^{-1})$ so 
$$
	U_2 \rightarrow \mathrm{E}[\alpha(x,y)] = 2 \Phi\left( -\frac{l^3 |\theta-\tsfrac{1}{2}| \sqrt{\tau}}{4} \right) \qquad \mbox{as $d\rightarrow \infty$},
$$
and
\begin{align*}
	|U_3| 
	&\leq (\sigma_i^2 + \mu_i^2)^{1/2} \left( 3 \frac{\tilde{g}_i^4}{\lambda_i^4} + 3 \frac{g_i^2}{\tilde{\lambda}_i^4} + 6 \frac{\tilde{g}_i^2 g_i}{\lambda_i^2 \tilde{\lambda}_i^2} \right)^{1/2} \\
	&= (\sigma_i^2 + \mu_i^2)^{1/2} \frac{\sqrt{3}}{4} \frac{h}{(1+\frac{\theta}{2}t_i)} \left( 4 + \frac{t_i}{1+\frac{\theta}{2}t_i} \right) \\
	&= \mathrm{o}(d^{-1/2-r}).
\end{align*}
Therefore, applying Theorem \ref{thm jumpsize} with the coordinate transformation $x \leftrightarrow V^{-1/2} x$ we obtain \eqref{eq jump2}.

\subsection{Proof of Theorem \ref{thm HMC1}}

The result will follow from Corollary \ref{lem equiv} with $G = (K^L)_{11}$ and $\Sigma = (K^L)_{12} V^{-1} (K^L)_{12}^T$ but we must first check that $G \Sigma$ is symmetric.  Define $\mathcal{V}$, $\mathcal{K}$ and $B$ as in Theorem \ref{thm iso hmc}.  Then $K = \mathcal{V} \mathcal{K} \mathcal{V}^{-1}$, so that $K^L = \mathcal{V} \mathcal{K}^L \mathcal{V}^{-1}$ and
$$
	(K^L)_{11} = V^{1/2} (\mathcal{K}^L)_{11} V^{-1/2} 
	\quad \mbox{and} \quad
	(K^L)_{12} = V^{1/2} (\mathcal{K}^L)_{12} V^{1/2}.
$$
Then
\begin{align*}
	G \Sigma = V^{1/2} (\mathcal{K}^L)_{11} (\mathcal{K}^L)_{12}^{2} V^{1/2} 
\end{align*}
which is symmetric because $V$ and $B$ are symmetric and $(\mathcal{K}^L)_{11}$ and $(\mathcal{K}^L)_{12}$ are polynomials of $B$.

\subsection{Proof of Corollary \ref{cor 15}}

First note that 
$$
	(I-(K^L)_{11}) \mathcal{A}^{-1} \beta = \left( SJ \left[ \twobyone{0}{\tsfrac{h}{2}b} \right] \right)_1,
$$
so we are required to show that
$$
	(I-(K^L)_{11}) A^{-1} b = \left( SJ \left[ \twobyone{0}{\tsfrac{h}{2}b} \right] \right)_1,
$$
which holds if
$$
	(I-K^L) \left[ \twobyone{A^{-1}b}{0} \right] = SJ \left[ \twobyone{0}{\tsfrac{h}{2}b} \right].
$$
Using $S = (I-K)^{-1}(I-K^L)$, we can equivalently show
$$
	(I-K) \left[ \twobyone{A^{-1}b}{0} \right] = J \left[ \twobyone{0}{\tsfrac{h}{2}b} \right],
$$
which is easy to check.

\subsection{Proof of Theorem \ref{thm uhmc conv}}

Define a spectral decomposition 
\begin{equation}
\label{eq specdecomp}
	V^{1/2}AV^{1/2} = Q \Lambda Q^T
\end{equation}
where $Q$ is an orthogonal matrix and $\Lambda = \operatorname{diag}(\lambda_1^2,\dotsc,\lambda_d^2)$ is a diagonal matrix of eigenvalues of $V^{1/2}AV^{1/2}$ ($VA$ is similar to $V^{1/2}AV^{1/2}$ so they have the same eigenvalues).  Also define $\mathcal{V}$ as in Theorem \ref{thm iso hmc} and 
$$
	\tilde{Q} = 	\left[ \twobytwo{Q}{0}{0}{Q} \right] \in \bbR^{2d \times 2d}.
$$
A similarity transform of $K$ is defined by
$$
	K = \mathcal{V} \tilde{Q} \tilde{K} \tilde{Q}^T \mathcal{V}^{-1}
\quad \mbox{
with} \quad
	\tilde{K} = \left[ \twobytwo{I - \tsfrac{h^2}{2}\Lambda}{hI}{-h\Lambda + \tsfrac{h^3}{4} \Lambda^2}{I - \tsfrac{h^2}{2}\Lambda} \right].
$$
Hence $K$ and $\tilde{K}$ have the same eigenvalues.  Moreover, $K^L = \mathcal{V} \tilde{Q} \tilde{K}^L \tilde{Q}^T \mathcal{V}^{-1}$ and it follows that 
$$
	(K^L)_{11} = V^{1/2} Q (\tilde{K}^L)_{11} Q^T V^{-1/2}.
$$
Thus $(K^L)_{11}$ and $(\tilde{K}^L)_{11}$ are similar.

Notice that $\tilde{K}$ is a $2\times2$ block matrix where each $d\times d$ block is diagonal.  Therefore, $\tilde{K}^L$ is also a $2\times2$ block matrix with diagonal blocks.  In particular, $(\tilde{K}^L)_{11}$ is a diagonal matrix, so the eigenvalues of $(\tilde{K}^L)_{11}$ are on the diagonal of $(\tilde{K}^L)_{11}$.  Moreover, 
$$
	[(\tilde{K}^L)_{11}]_{ii} = (k_i^L)_{11}
$$
where $[(\tilde{K}^L)_{11}]_{ii}$ is the $i^{\mathrm{th}}$ diagonal entry of $(\tilde{K}^L)_{11}$,  $(k_i^L)_{11}$ is the $(1,1)$ entry of the matrix $k_i^L \in \bbR^{2\times2}$, and $k_i \in \bbR^{2\times2}$ is defined by
$$
	k_i = \left[ \twobytwo{(\tilde{K}_{11})_{ii}}{(\tilde{K}_{12})_{ii}}{(\tilde{K}_{21})_{ii}}{(\tilde{K}_{22})_{ii}} \right]
	= \left[ \twobytwo{1-\tsfrac{h^2}{2}\lambda_i^2}{h}{-h\lambda_i^2 + \tsfrac{h^3}{4}\lambda_i^4}{1-\tsfrac{h^2}{2}\lambda_i^2} \right].
$$  
The matrix $k_i$ can be factorized
$$
	k_i = \left[ \twobytwo{1}{0}{0}{a} \right] \left[ \twobytwo{\cos(\theta_i)}{-\sin(\theta_i)}{\sin(\theta_i)}{\cos(\theta_i)} \right]
	\left[ \twobytwo{1}{0}{0}{a^{-1}} \right]
$$
where $a = \lambda \sqrt{1 - \tsfrac{h^2}{4}\lambda_i^2}$ and $\theta_i = -\cos( 1-\tsfrac{h^2}{2}\lambda_i^2)$.  Therefore,
$$
	k_i^L = \left[ \twobytwo{1}{0}{0}{a} \right] \left[ \twobytwo{\cos(L \theta_i)}{-\sin(L \theta_i)}{\sin(L \theta_i)}{\cos(L \theta_i)} \right]
	\left[ \twobytwo{1}{0}{0}{a^{-1}} \right]
$$
and hence
$$
	[(\tilde{K}^L)_{11}]_{ii} = (k_i^L)_{11} = \cos(L\theta_i).
$$

\subsection{Proof of Theorem \ref{thm hmc conv}}

\begin{proof}
Theorem \ref{thm iso hmc} implies that it is equivalent to study the MH algorithm with target and proposal given by \eqref{eq hmchat1}, and since $(\mathcal{K}^L)_{11}$ and $(\mathcal{K}^L)_{12}$ are functions of $B$ we can apply Theorems \ref{thm accept} and \ref{thm jumpsize}.  Using the spectral decomposition \eqref{eq specdecomp} note that
$$
	(\mathcal{K}^L)_{ij} = Q (\tilde{K}^L)_{ij} Q^T \qquad \mbox{for $i,j=1,2$}.
$$
where $\tilde{K}$ is defined in the proof of Theorem \ref{thm uhmc conv}, and where it is shown that $(\tilde{K}^L)_{ij}$ is diagonal and
$$
	[(\tilde{K}^L)_{11}]_{ii} = \cos (L \theta_i)
$$
where $\theta_i = -\cos^{-1}(1-\tsfrac{h^2}{2}\lambda_i^2)$.  Similarly,
$$
	[(\tilde{K}^L)_{12}]_{ii} = -a_i^{-1} \sin (L \theta_i)
$$
where $a_i = \lambda_i \sqrt{1-\tsfrac{h^2}{4}\lambda_i^2}$.  Moreover, $\mathcal{A} = Q (\tilde{K}^L)_{12}^2 (I-(\tilde{K}^L)_{11}^2) Q^T$ so that $\tilde{\lambda}^2_i = -a_i^2$ and if we let $t_i = h^2 \lambda_i^2$, then 
\begin{align*}
	\tilde{\lambda}_i^2 &= \lambda_i^2 (1 - \tsfrac{1}{4}t_i), & G_i &= \cos(L\theta_i), & \tilde{g}_i &= 1 - \cos(L\theta_i), & g_i &= \sin^2(L \theta_i), \\
	r_i &= \tsfrac{1}{4} t_i, & 1+\tilde{r}_i &= \tsfrac{1}{1-\tsfrac{1}{4}t_i}, & \hat{r}_i &= 0.
\end{align*}
Note that we used Corollary \ref{cor 15} to show $\hat{r}_i = 0$.  Then
\begin{align*}
	T_{0i} &= T_{1i} = T_{2i} = 0, \\
	T_{3i} &= \tsfrac{1}{8} t_i \sin^2(L\theta_i), &
	T_{4i} &= - \frac{\tsfrac{1}{8} t_i \sin^2(L\theta_i)}{1-\tsfrac{1}{4}t_i}, &
	T_{5i} &= - \frac{ \tsfrac{1}{8} t_i \sin(2L\theta_i)}{\sqrt{1-\tsfrac{1}{4}t_i}}.
\end{align*}
The trigonmetric expansion $\cos^{-1}(1-z) = \sqrt{2z} + \mathcal{O}(z^{3/2})$, $t_i = l^2 d^{-1/2} (\tsfrac{i}{d})^{2\kappa} \\= \mathrm{o}(d^{-1/2})$ and defining $T'$ such that $L = \tsfrac{T'}{h}$ implies there exists a function $T''(d)$ such that
$$
	L \theta_i = \lambda_i (T' + \mathrm{o}(d^{-1/2})) = (\tsfrac{i}{d})^\kappa (T' d^\kappa + \mathrm{o}(d^{\kappa-1/2}) =: (\tsfrac{i}{d})^\kappa T''(d),
$$
and $T''(d) \rightarrow \infty$ as $d \rightarrow \infty$.

To apply Theorems \ref{thm accept} and \ref{thm jumpsize} we need to check \eqref{eq Tcond}.  For some $h > 0$ we find
\begin{align*}
	\lim_{d \rightarrow \infty} \frac{ \sum_{i=1}^d |T_{3i}|^{2+\delta} }{ \left( \sum_{i=1}^d |T_{3i}|^2 \right)^{1+\delta/2} }
	&= \lim_{d \rightarrow \infty} \frac{ \sum_{i=1}^d |t_i \sin^2(L\theta_i) |^{2+\delta} }{ \left( \sum_{i=1}^d |t_i \sin^2(L\theta_i)|^2 \right)^{1+\delta/2} } \\
	&= \lim_{d \rightarrow \infty} \frac{ d^{-\delta/2} \sum_{i=1}^d d^{-1} |(\tsfrac{i}{d})^{2\kappa} \sin^2((\tsfrac{i}{d})^\kappa T''(d)) |^{2+\delta} }{ \left( \sum_{i=1}^d d^{-1} |(\tsfrac{i}{d})^{2\kappa} \sin^2((\tsfrac{i}{d})^\kappa T''(d)) |^2 \right)^{1+\delta/2} } \\
	&= \lim_{d \rightarrow \infty} d^{-\delta/2} \frac{ \int_0^1 |z^{2\kappa} \sin^2(z^\kappa T'')|^{2+\delta} \mathrm{d}z }{ \left( \int_0^1 |z^{2\kappa} \sin^2(z^\kappa T'')|^2 \mathrm{d}z \right)^{1+\delta/2}} \\
	&= \lim_{d \rightarrow \infty} d^{-\delta/2} \frac{ \int_0^1 |z^{2\kappa} |^{2+\delta} \mathrm{d}z }{ \left( \int_0^1 |z^{2\kappa} |^2 \mathrm{d}z  \right)^{1+\delta/2}} = 0.
\end{align*}
Similar arguments verify \eqref{eq Tcond} for $T_{4i}$ and $T_{5i}$.  Now we can apply Theorem \ref{thm accept} with
\begin{align*}
	\mu 
	&= \lim_{d \rightarrow \infty} \sum_{i=1}^d T_{3i} + T_{4i} \\
	&= \lim_{d \rightarrow \infty} -\frac{1}{32} \sum_{i=1}^d \frac{t_i^2 \sin^2(L\theta_i)}{1-\tsfrac{1}{4}t_i} \\
	&= \lim_{d \rightarrow \infty} -\frac{1}{32} \sum_{i=1}^d t_i^2 \sin^2(L\theta_i) \\
	&= \lim_{d \rightarrow \infty} -\frac{l^4}{32} \sum_{i=1}^d d^{-1} (\tsfrac{i}{d})^{4\kappa} \sin^2((\tsfrac{i}{d})^\kappa T''(d)) \\
	&= \lim_{d \rightarrow \infty} -\frac{l^4}{32} \int_0^1 z^{4\kappa} \sin^2(z^\kappa T''(d)) \mathrm{d}z \\
	&= -\frac{l^4}{32} \frac{1}{2\pi} \int_0^{2\pi} \sin^2(z') \mathrm{d}z' \int_0^1 z^{4\kappa} \mathrm{d}z \\
	&= -\frac{l^4}{64(1+4\kappa)},
\end{align*}
and similarly,
\begin{align*}
	\sigma^2 &= \lim_{d \rightarrow \infty} \sum_{i=1}^d 2 T_{3i}^2 + 2 T_{4i}^2 + T_{5i}^2 \\
	&= \lim_{d \rightarrow \infty} \sum_{i=1}^d \frac{1}{32} t_i^2 \sin^4(L\theta_i) + \frac{1}{32} \frac{t_i^2 \sin^4(L\theta_i)}{(1-\tsfrac{1}{4}t_i)^2} + \frac{1}{64} \frac{t_i^2 \sin^2(2L\theta_i)}{1-\tsfrac{1}{4}t_i} \\
	&= \lim_{d \rightarrow \infty} \sum_{i=1}^d \frac{1}{16} t_i^2 \sin^4(L\theta_i) + \frac{1}{64} t_i^2 \sin^2(2L\theta_i) \\
	&= \lim_{d \rightarrow \infty} \sum_{i=1}^d \frac{l^4}{16} d^{-1} (\tsfrac{i}{d})^{4\kappa} \sin^4((\tsfrac{i}{d})^\kappa T''(d)) + \frac{l^4}{64} d^{-1} (\tsfrac{i}{d})^{4\kappa} \sin^2(2(\tsfrac{i}{d})^\kappa T''(d)) \\
	&= \lim_{d \rightarrow \infty} \frac{l^4}{16} \int_0^1 z^{4\kappa} \sin^4(z^\kappa T'') \mathrm{d}z + \frac{l^4}{64} \int_0^1 z^{4\kappa} \sin^2(2z^{\kappa}T'') \mathrm{d}z \\
	&= \frac{l^4}{16} \left( \frac{1}{2\pi} \int_0^{2\pi} \!\! \sin^4(z') \mathrm{d}z' \right) \int_0^1 \!\! z^{4\kappa} \mathrm{d}z + \frac{l^4}{64} \left( \frac{1}{2\pi} \int_0^{2\pi}\!\! \sin^2(z') \mathrm{d}z' \right) \int_0^1 \!\! z^{4\kappa} \mathrm{d}z \\
	&= \frac{l^4}{32(1+4\kappa)}.
\end{align*}
Hence $\frac{\mu}{\sigma} = -\sigma - \frac{\mu}{\sigma} = - \frac{l^2}{8 \sqrt{2} \sqrt{1+4\kappa}}$ and $\mu+\sigma^2/2 = 0$, so from Theorem \ref{thm accept} we obtain \eqref{eq hmcexpect}.

For the expected jump size, we apply Theorem \ref{thm jumpsize} with
$$
	U_1 = \frac{\tilde{g}_i^2}{\lambda_i^2} + \frac{g_i}{\tilde{\lambda}_i^2} = \frac{(1-\cos(L\theta_i))^2}{\lambda_i^2} + \frac{\sin^2(L\theta_i)}{\lambda_i^2 (1-\frac{1}{4}t_i)} \rightarrow \frac{2(1-\cos(\lambda_i T'))}{\lambda_i^2} + \mathrm{o}(d^{-1/2})
$$
as $d\rightarrow \infty$.  Also, it is straightforward to show that $\mu_i = \mathrm{o}(d^{-1})$ and $\sigma_i^2 = \mathrm{o}(d^{-1})$.  Hence
$$
	U_2 \rightarrow \mathrm{E}[\alpha(x,y)] = 2 \Phi\left( -\frac{l^2}{8\sqrt{2}\sqrt{1+4\kappa}} \right)
$$
as $d \rightarrow \infty$, and 
\begin{align*}
	|U_3| &\leq (\sigma_i^2 + \mu_i^2)^{1/2} \left( 3 \frac{\tilde{g}_i^4}{\lambda_i^4} + 3 \frac{g_i^2}{\tilde{\lambda}_i^4} + 6 \frac{\tilde{g}_i^2 g_i}{\lambda_i^2 \tilde{\lambda}_i^2} \right)^{1/2} \\
	&= (\sigma_i^2 + \mu_i^2)^{1/2} \frac{\sqrt{3}}{\lambda_i^2} \left( {\scriptstyle (1-\cos(L\theta_i))^4 + \frac{\sin^4(L\theta_i)}{(1-\tsfrac{1}{4}t_i)^2} + 2 \frac{(1-\cos(L\theta_i))^2 \sin^2(L\theta_i)}{1-\tsfrac{1}{4}t_i} } \right)^{1/2} \\
	&= \mathrm{o}(d^{-1/2}).
\end{align*}
Therefore, we obtain \eqref{eq hmcjump}.
\end{proof}

\bibliographystyle{imsart-number}
\bibliography{paper8bib} 

\begin{thebibliography}{24}
\bibitem{FP2015}
\begin{bunpublished}[author]
\bauthor{\bsnm{Fox},~\bfnm{Colin}\binits{C.}} \AND
  \bauthor{\bsnm{Parker},~\bfnm{Albert}\binits{A.}}
\btitle{Accelerated {G}ibbs sampling of normal distributions using matrix
  splittings and polynomials}.
\bnote{not yet published}.
\end{bunpublished}
\endbibitem
\end{thebibliography}

\end{document}